\documentclass[12pt]{article}
\usepackage[utf8]{inputenc}
\usepackage{amsmath, amsfonts, amssymb, enumitem, graphicx, subcaption, wrapfig, stmaryrd, mathtools, xcolor}

\usepackage[left=2.5cm,right=2.5cm,top=2.5cm,bottom=4cm]{geometry}
\usepackage{dsfont}
\usepackage{lmodern}
\usepackage[T1]{fontenc}
\usepackage{pict2e}
\usepackage{biblatex}
\usepackage[standard]{ntheorem}
\usepackage{hyperref}
\usepackage{titlesec}
\titlespacing{\section}{15pt}{*4}{*1.5}
\titlespacing{\subsection}{30pt}{*4}{*1.5}
\titlespacing{\subsubsection}{45pt}{*4}{*1.5}

\newcommand{\fonction}[5]{#1:{\left\{\begin{array}{lcl}
#2 & \longrightarrow & #3 \\
 #4 & \longmapsto & #5 \end{array}\right.}}

\newcommand{\itg}{\displaystyle\int }

\newcommand{\R}{\mathbb{R}}
\newcommand{\C}{\mathbb{C}}

\newcommand{\cont}{\mathcal{C}}

\newcommand{\pinf}{^{\infty}}

\newcommand{\s}{\mathcal{S}}
\newcommand{\Spd}{\mathbb{S}^{d-1}}

\newcommand{\M}{\mathcal{M}}

\author{\textsc{ Corentin Gentil}\footnote{corentin.gentil@ens.fr},  \textsc{ Côme Tabary}\footnote{come.tabary@ens.fr} }
\date{January 2024}
\title{Strichartz estimates for geophysical fluid equations using Fourier restriction theory}

\begin{document}
\maketitle

\section*{Abstract}

We prove Strichartz estimates for the semigroups associated to stratified and/or rotating inviscid geophysical fluids using Fourier restriction theory. We prove new results for rotating stratified fluids, and recover results from \cite{kohleetakada} for rotation only, and from \cite{lee_takada} for stratification only. Our restriction estimates are obtained by the slicing method \cite{nicola}, which relies on the well-known Tomas-Stein theorem for 2-dimensional spheres. To our knowledge, such a method has never been used in this setting. Moreover, when the fluid is stratified, our approach yields sharp estimates, showing that the slicing method captures all the available curvature of the surfaces of interest.

\section{Introduction}

\subsection{Main result}\label{sec:main}

The \emph{primitive system} describes the motion of incompressible geophysical fluids accounting for both the Earth's rotation and the stratification of the fluid \cite{charve} \cite{iwabuchi} \cite{pedlosky}. We consider in this paper its inviscid limit:
\begin{equation}
    \label{pe}
    \left\{\begin{array}{rcl}
\partial_t u + (u \cdot \nabla) u + \varepsilon^{-1} e_3 \wedge u + F^{-1} \varepsilon^{-1} \theta e_3 &=& - \varepsilon^{-1} \nabla p\\
\partial_t \theta + (u\cdot\nabla)\theta - F^{-1} \varepsilon^{-1} u_3 &=&0\\
\nabla \cdot u &=&0\\
(u\vert_{t=0},\theta\vert_{t=0}) &=& (u^0,\theta^0),
 \end{array}\right.
\end{equation} 
Here $u=(u_1,u_2,u_3)^T$ is the fluid velocity, $p$ the pressure, and $\theta$ the temperature. The effects of stratification and rotation are quantified respectively by the physical Froude number $\varepsilon F$, comparing the advection term to the gravitational forces exerted on the fluid, and the Rossby number $\varepsilon$, characterizing the importance of rotation over advection.  We follow the notations of \cite{charve} and call $F$ the Froude number in the sequel.

The above system exhibits a dispersive behaviour closely related to the semigroup generated by the fluid rotation and stratification, which is defined by the oscillatory integral:
\begin{equation}
    \label{op_pri}
e^{\pm \frac 1 \varepsilon i t p_F(D)}g : (t,x) \longmapsto \itg_{\R^3} e^{ ix\cdot \xi \pm \frac 1 \varepsilon i t p_F(\xi)}\hat g (\xi) d\xi,
\end{equation}
where $\hat g$ is the Fourier transform of $g$, and the phase is given by $$p_F(\xi) = \dfrac{ \vert \xi \vert_F}{F \vert \xi \vert}\ , \ \ \ \ \ \vert \xi \vert_F=\vert (\xi_1, \xi_2, F \xi_3)\vert.$$
The expression of this operator and its connection with the primitive system are derived in Section 3.1. The main results of this paper are the following Strichartz estimates:
 
\begin{theorem}
\label{th:pri}
    There exists a universal constant $C>0$ such that the following homogeneous Strichartz estimate holds for all $g\in \dot{H}^1 (\R^3)$:
\begin{equation}
\label{esti_pri}
    \Vert e^{\pm \frac 1 \varepsilon i t p_F(D)}g \Vert_{L^6_{t} L^6_{x}} \leq C \varepsilon^{\frac{1}{6}} \frac{F^{\frac 1 2}}{( F-1 )^{\frac 1 6}} \Vert g \Vert_{\dot H^1}.
\end{equation} 
Moreover, for all $f\in L^1_t\dot H^1_x$, the following inhomogeneous estimate holds:
\begin{equation}
\label{duhamel_pri}
    \left\Vert \itg_0^t e^{\pm \frac 1 \varepsilon i (t-t') p_F(D)}f(t') dt' \right\Vert_{L^6_{t} L^6_{x}} \leq C \varepsilon^{\frac{1}{6}} \frac{F^{\frac 1 2}}{( F-1 )^{\frac 1 6}} \Vert f \Vert_{L^1_t\dot H^1_{x}}.
\end{equation} 
\end{theorem}
The estimate \eqref{esti_pri} is sharp, in the following sense:
\begin{proposition}
    \label{prop:sharp}
    Let $p,\ q \geq 2$, and $s \in \R$. If the following estimate 
    \begin{equation}
    \label{sharppri}
        \Vert e^{\pm \frac 1 \varepsilon i t p_F(D)}g \Vert_{L^p_t L^q_x}\leq C \Vert g \Vert_{\dot H^s}
    \end{equation} holds for all $g\in \dot H^s$, then necessarily $p,\ q$ and $s$ must satisfy the scaling conditions:
    \begin{equation}
    \label{scaling} \frac 2 p + \frac 1 q \leq \frac 1 2, \ \ \ \ \ \ \text{and} \ \ \ \ \ \ s=3\bigg(\frac 1 2 - \frac 1 q\bigg).\end{equation}
\end{proposition}
Moreover, one cannot find an estimate of the type of \eqref{sharppri} with an $L^p_tL^q_x$ norm on the left-hand side, and an $L^r_x$ norm on the right-hand side, with $p,q,r \in (2,\infty)$, justifying the resort to Sobolev regularity of the initial data. 

The proof we give of Theorem \ref{th:pri} resorts to restriction theory, and more specifically to the slicing method \cite{nicola}, therefore it differs from the dispersive approach of \cite{dut_2}, \cite{charve}, \cite{kohleetakada} and \cite{lee_takada}. A more in-depth comparison to the literature is made in Section 1.5. In Section 1.6 we apply the slicing method to other geophysical systems to recover already known estimates but with this new approach.

\subsection{Strichartz estimates}

Over the last decades, Strichartz estimates have become a staple in the analysis of evolution partial differential equations, see for instance \cite{bcg}, \cite{velo}, \cite{stri} and the references therein. Indeed, such estimates have proved to be a key tool in the derivation of the well-posedness of many systems, as the wave or the Schrödinger equation.

Considering a linear evolution PDE, a Strichartz estimate is a bound for the space-time norm of the solution $u$ by the norm of the initial datum $u_0$, that is to say an inequality of the type
$$ \Vert u \Vert_{L^p_tL^q_x} \leq C \Vert u_0 \Vert_{L^r_x},$$
where $ \Vert u \Vert_{L^p_t L^q_x}$ is the standard Lebesgue space-time norm on $L^p(\R,\ L^q(\R^d))$, and $p, q, r$ satisfy some scaling condition.

Strichartz inequalities are commonly obtained using the dispersive nature of the equation. The linear Schrödinger equation provides an explicit setting to illustrate the derivation and application of Strichartz estimates. Consider the system
\begin{equation}
\label{schrodingerhom}
{{\left\{\begin{array}{lrcl}
i\partial_t u + \Delta u =0, \ \ \ \text{for   } (t,x)\in\R^+ \times \R^d \\
u\vert_{t=0} = u_0, \end{array}\right.}}
\end{equation}
to which the solution is, for $t>0$, 
\begin{equation}
    \label{convolution}
    u(t,\cdot) = \frac{1}{(4\pi it)^{\frac d 2}} e^{i\frac{\vert \cdot \vert^2}{4t} } \star u_0.
\end{equation} 

Young's inequality applied to this explicit formula leads to the following dispersive estimate:
$$ \Vert u(t,\cdot) \Vert_{L\pinf_x(\R^d)} \leq C \frac{1}{(4\pi t)^{\frac d 2}} \Vert u_0 \Vert_{L^{1}_x(\R^d)},$$
for $ t >0$. This inequality expresses that, even though the total mass $\Vert u_0 \Vert_{L^{2}_x(\R^d)}$ is conserved by the Schrödinger semigroup, the propagation at different speeds of different frequency modes leads to a decay in $L\pinf_x$ norm.

The $TT^*$ argument (which states that the continuity of an operator $T$ is equivalent to that of its dual $T^*$ or of the operator $TT^*$, see \cite{velo}, \cite{kato}, \cite{keel}) allows one to convert a dispersive inequality into a Strichartz estimate (see Chapter 8 in \cite{bcg} on Strichartz estimates and applications to semilinear dispersive equations). For the homogeneous Schrödinger equation \eqref{schrodingerhom}, it writes: 
$$\Vert u \Vert_{L^q_t L^r_x} \leq C \Vert u_0 \Vert_{L^2_x},$$
with $q\in(2,\infty]$ and $r$ satisfying the scaling condition: 
\begin{equation}
    \label{scaling2}
    \frac 2 q + \frac d r = \frac d 2 \cdotp
\end{equation} 
Such a scaling stems directly from the equation itself: indeed, if $u$ is a solution of the Schrödinger equation, then $u_\lambda=u(\lambda^2 \cdot, \lambda \ \cdot)$ also solves \eqref{schrodingerhom} with suitably rescaled initial condition. Thereby, it is easy to see that the condition on the exponents translates the necessity of the Strichartz estimate to be scale invariant as well.

One of the applications of such estimates is the proof of existence of solutions in inhomogeneous settings, see the reference article \cite{keel}. Indeed, provided a forcing term $f\in L^{q'}_t L^{r'}_x$ with $(q',r')$ the dual exponents of $(q,r)$ satisfying \eqref{scaling2}, there exists a solution $u$ to the equation
$${{\left\{\begin{array}{lrcl}
i\partial_t u + \Delta u = f, \ \ \ \text{for   } (t,x)\in\R^+ \times \R^d \\
u\vert_{t=0} = 0, \end{array}\right.}}$$
which verifies the following inhomogeneous Strichartz estimate for any pair $(p,s)$, again satisfying \eqref{scaling2}:
$$ \Vert u \Vert_{L^{p}_t L^{s}_x} \leq C \Vert f \Vert_{L^{q'}_t L^{r'}_x}.$$
Inhomogeneous estimates in turn offer ways to treat nonlinear problems, as it is done in \cite{velo} for the nonlinearity $f= \pm u \vert u \vert^{p-1}$, by fixed points arguments.

The application of Strichartz estimates to the particular study of geophysical fluids will be overviewed in Section 1.5 of this introduction.

\subsection{Restriction theory}

For any function $f$ in the Lebesgue space $L^1(\R^d)$, its Fourier transform $\hat{f}$, defined for $\xi \in \R^d$ as
$$\hat f (\xi) = \int_{\R^d}e^{-ix\cdot\xi}f(x)dx,$$
is continuous, so the restriction of $\hat f$ to any surface of $\R^d$ is well defined. If $f$ is not integrable but only lies in some $L^p$ space, its Fourier transform can still be defined and the Hausdorff-Young inequality implies that if $1\leq p \leq 2$, then $\hat f$ is in the dual space $L^{p'}$, where $\frac{1}{p}+\frac{1}{p'}=1$. Thus $\hat f$ is only defined almost everywhere, and it is not clear whether it is possible to meaningfully restrict $\hat{f}$ to a hypersurface $S$ of $\R^d$, since $S$ is a set of measure zero. The restriction problem asks for which exponents $p$ and $q$ a restriction estimate of the form
$$ \Vert \hat{f}\vert_{S}\Vert_{L^q(S,d\sigma)} \leq C \Vert f \Vert_{L^p(\R^d)}$$ holds for all functions $f$ in the Schwartz class $\s(\R^d)$, where $d\sigma$ is a measure on $S$. Such estimates indeed allow one to define the restriction of the Fourier transform $f\mapsto \hat{f}\vert_{S}$ as a bounded operator from $L^p(\R^d)$ to $L^q(S,d\sigma)$.

Tomas and Stein discovered \cite{tsb}, \cite{ts} that this is possible for surfaces which are "sufficiently curved", a result that is now known as the Tomas-Stein theorem:

\begin{theorem}[\cite{tsb}, \cite{ts}] \label{tomasstein}
Let $d\geq 2$ and $p_0 = \frac{2(d+1)}{d+3}$. Let $S$ be a smooth and compact hypersurface in $\R^d$ endowed with a smooth measure $d\sigma$. Suppose $S$ has non vanishing Gaussian curvature at every point. Then for all $p\in [1, p_0]$, there exists a positive constant $C$ such that for any $f\in\s(\R^d)$:
$$ \Vert \hat f \vert_{S} \Vert_{L^2(S,d\sigma)} \leq C \Vert f \Vert_{L^p(\R^d)}.$$
\end{theorem}

The canonical example of such a surface is the sphere $\Spd$ of $\R^d$. When the hypothesis of curvature and compactness is weakened, obtaining restriction estimates does not become impossible, but the range of possible exponents usually gets smaller, as we will illustrate with the surfaces considered in this paper. A first step towards understanding the necessity of curvature is to consider a simple counter-example (found for instance in \cite{tao}) preventing any non-trivial restriction estimate on a hyperplane (endowed with the natural surface measure): consider $\varphi\in \s (\R^{d-1})$ a Schwartz function, and define $h:\R^d\mapsto\C$ by:
$$h(x)=\dfrac{\varphi(x')}{\sqrt{1+x_{d}^{2}}}, \; \; \; x=(x',x_{d}).$$
Then $h$ belongs to every $L^p$ space for $p>1$, but its Fourier transform is not well defined on the hyperplane ${\xi_d=0}$, or even on any bounded subset of this hyperplane. Curvature is thus mandatory to obtain restriction estimates. 

However, lack of compactness can sometimes be dealt with using scaling arguments, provided the surface possesses some symmetry. It is for instance the case of the paraboloid $$ \mathcal{P} = \{ (\xi ,\vert \xi \vert^2),\ \xi\in \R^{d-1} \} \subset \R^d, $$
which is non-compact but still belongs to the class of surfaces for which the Tomas-Stein theorem stated above applies. The two main reasons for this are the fact that the paraboloid is strongly symmetric, and the scaling satisfied by the paraboloid. Indeed, if we stretch the space by a factor $\lambda >0$ on the first $d-1$ coordinates and by a factor $\lambda^2$ in the last coordinate, the paraboloid $\mathcal{P}$ remains unchanged. 

\emph{In this paper, we will exclusively work with non-compact surfaces, so none of them satisfy the Tomas-Stein theorem. We will therefore be led to consider spaces with more sophisticated measures. }

Restriction theory has many applications to other topics, from number theory to the Bochner-Riesz and Kakeya conjectures \cite{tao}, but our point of interest here is its connection to Strichartz estimates. Indeed, it provides a pathway to prove Strichartz inequalities without relying on the dispersive estimate we presented above.

As it was shown by Strichartz in \cite{stri}, solutions to linear PDEs can be written as Fourier transforms on surfaces. Let us consider this time the free wave equation
\begin{equation}
{\left\{\begin{array}{lrcl}
\partial^2_{tt} u (t,x)- \Delta u(t,x) = 0, \ \ \ \ \ \ \ \ \ \text{for   } (t,x)\in\R^+ \times \R^d\\
u|_{t=0} = f\ \ \ \ \text{and}\ \ \ \ \partial_t u|_{t=0} = g. \end{array}\right.}
\end{equation}
Resolution in Fourier space gives a representation formula for the solution:
$$ u(t,x) = \itg_{\R^d_{\xi} } e^{it\vert \xi\vert + i x \cdot \xi} \gamma_+(\xi) d\xi + \itg_{\R^d_{\xi}} e^{-it\vert \xi\vert + i x \cdot \xi} \gamma_-(\xi) d\xi,$$
where $\gamma_{\pm} (\xi) =\frac{1}{2} \bigg(  \hat f(\xi) \pm \frac{\hat g(\xi)}{i\vert\xi\vert} \bigg)$. From this formula, one could apply the stationary phase argument and obtain a dispersive estimate, similar to the one obtained directly from \eqref{convolution} in the case of the Schrodinger equation. Details of this approach can be found in \cite{bcg} and references therein.

The solution can thus be seen as the sum of two inverse Fourier transforms, each one on a cone $\mathcal{C}_{\pm} = \{ \zeta = (\pm \vert\xi\vert ,\xi),\ \xi \in \R^d \}$. Indeed, this can be rewritten as

$$ u(t,x) = \itg_{\mathcal{C_+}} e^{i (t,x) \cdot \zeta } \ \ \Tilde{\gamma}_+(\zeta)  \ d\mu(\zeta) + \itg_{\mathcal{C_-}} e^{i (t,x) \cdot \zeta }\ \  \Tilde{\gamma}_-(\zeta)  \ d\mu(\zeta), $$
by introducing the pullback of $\gamma$ onto the cone by the canonical projection $( \pm \vert\xi\vert ,\xi) \mapsto \xi$, that is $\Tilde{\gamma}_{\pm}(\pm \vert \xi \vert, \xi) = \gamma_{\pm}(\xi) $. The measure $d\mu(\pm \vert \xi \vert, \xi ) = d\xi$ is as well the pullback of the Lebesgue measure. Written this way, this inverse Fourier transform on the cone is exactly the dual of the restriction operator $f\mapsto \hat{f}\vert_{\mathcal{C}_{\pm}}$ applied to $\Tilde{\gamma}_{\pm}$ (justifying why such inverse Fourier transforms are referred to as extension operators).

This is the connection inciting our resort to restriction theory:
indeed, suppose that we have the following restriction estimate on the cone: 

$$\left\Vert  {\hat v}\vert_{\cont_{\pm}}\  {\vert \xi \vert^{-\frac{1}{2}}} \right\Vert_{L^2(\cont_{\pm},\ d\mu)} \leq C \Vert v \Vert_{L^{p_0}(\R^{d+1})}, $$
where $p_0 = \frac{2(d+1)}{d+3}$. In its dual form, it writes, for $w\in L^2(\cont_{\pm}, d\mu)$:

$$ \left\Vert \mathcal{F}^{-1} (\vert \xi \vert^{-\frac{1}{2}} w \ d\mu) \right\Vert_{L^{p_0'}(\R^{d+1})} \leq C \Vert w \Vert_{L^2(\cont_{\pm}, d\mu)}   .$$
Taking $w = \Tilde{\gamma}_{\pm} \vert \xi \vert^{\frac{1}{2}}$ in the above inequality and using that $$ \left\Vert \Tilde{\gamma}_{\pm} \vert \xi \vert^{\frac{1}{2}} \right\Vert_{L^2(\cont_{\pm}, d\mu)} = \Vert \mathcal{F}^{-1} \gamma_{\pm} \Vert_{\dot{H}^{\frac{1}{2}}(\R^d)} \leq C \left( \Vert f \Vert_{\dot{H}^{\frac{1}{2}}(\R^d)}  + \Vert g \Vert_{\dot{H}^{-\frac{1}{2}}(\R^d)} \right),  $$
to bound the right-hand side, we obtain the Strichartz estimate
$$ \Vert u \Vert_{L^{p_0'}_{t,x}} \leq C \left( \Vert f \Vert_{\dot{H}^{\frac{1}{2}}(\R^d)}  + \Vert g \Vert_{\dot{H}^{-\frac{1}{2}}(\R^d)} \right).$$

Hence, through fairly simple rewriting of the representation formula, the obtention of the Strichartz estimate above reduces to a restriction estimate on a conic surface. Although the surfaces will be significantly more complex, the same reasoning will hold in the study of geophysical fluids, and it is at the core of our approach.

\subsection{Geophysical fluids}

In this article, we study systems of PDEs governing the motion of geophysical fluids. These equations arise naturally as soon as one takes into account the scales relevant to the study of the atmosphere or the ocean: careful modeling choices must be carried out to include adequate forcing terms and dimensionless numbers correctly reflecting the real contribution of each term. Taking the example of the ocean, its behaviour can be described by the incompressible Navier-Stokes equations with additional forces, that is
\begin{equation}\label{ns}
     {\left\{\begin{array}{rcl}
\partial_{t} u + u \cdot \nabla u + \Omega \wedge u  - \nu \Delta u & = & - \frac{1}{\rho}\nabla p - g \\
\nabla \cdot u & = & 0
\end{array}\right.} 
\end{equation}
where $u$ is a 3D vector field standing for the fluid velocity, $p$ is the pressure, $\rho$ the fluid density, and $\nu$ the diffusive parameter. The two exterior forces are the gravitational force and the Coriolis force caused by Earth's rotation, namely $g$ the gravity field, and $\Omega$ the rotation speed of the referential.

And yet, as mentioned above, this universal equation is never the one implemented in practice in the climate models. Indeed, there are many physically meaningful approximations that reduce the complexity of System \eqref{ns}, such as  the hydrostatic approximation which is used in most climate models (see  \cite{pedlosky}). 

Another example is the importance of the Coriolis force. For the time scale of the large oceanic circulation, a characteristic displacement of $10^2$ km takes approximately $10^2$ days at 1cm/s (the characteristic speed of ocean currents), that means that the Earth has rotated a hundred times during this displacement. Therefore, the rotation term can be considered to be a penalisation term, that is to say the contribution of this term is so overwhelming that the solution must lie in the kernel of the associated operator. It leads to a rigorous proof of the Taylor-Proudman theorem (\cite{cdgg}), which was already well-known by physicists (see \cite{pedlosky}).

Roughly speaking, it states that a fast rotating fluid has a horizontal speed independent of the vertical coordinate and no vertical speed. 

Many physical considerations must be taken into account to derive from the Navier-Stokes equation a system more suited to the description of the oceans or the atmosphere at large scales. The penalisation of the rotation term is one of them, but we could also mention the aspect ratio, and the effect of the density governed by the temperature, the pressure for the air and the salinity for the ocean, that will give birth to a strong stratification with different associated effects. In particular, Sections 3 and 4 will be devoted to stratified fluids.

\subsection{Standard method for geophysical fluid equations}

The study of the incompressible Navier-Stokes equations heavily relies on energy estimates, which express the physical reality of conservation of the total energy of the fluid. Weak solutions (in the sense of Leray \cite{leray}) exhibit non-increasing —rather than constant— energy, and such estimates are key to prove existence of Leray solutions. Geophysical fluids typically obey the Navier-Stokes equations with additional terms describing the aforementioned physical effects of rotation and fluid stratification. To illustrate our discussion, consider the simpler example of the Navier-Stokes system for rotating fluids:
\begin{equation}
\label{nsce}
        {\left\{\begin{array}{rcl}
\partial_t u + u \cdot \nabla u - \nu \Delta u + \dfrac{e_3 \wedge u}{\varepsilon}  &=& - \nabla p \\
\nabla \cdot u &=& 0 \\
u(0,\cdot)&=&u_0.
\end{array}\right.}
\end{equation}
 Here, $u$ is the velocity field, $p$ the pressure, and $\frac 1 \varepsilon $ is the rotation speed directed along the $x_3$-axis, appearing in front of the Coriolis force $e_3 \wedge u$. Geophysical terms are skew-symmetric, entailing that they cancel out in energy estimates (see the discussion in \cite{cheminpenal}). On the one hand, this means energy is still conserved so the results on Leray solutions can be easily transposed to the case of geophysical systems. On the other hand, the energy estimates completely fail to capture the additional regularity offered by these terms. The goal of this subsection is to present how geophysical fluids are approached from a mathematical viewpoint and how new and enhanced results can be proved compared to the Navier-Stokes system. 

The usual approach is to get rid of the unknown pressure by applying the Leray projector $\mathbb P$ on divergence-free vector fields, and to treat the non-linear advection $u\cdot \nabla u$ as an exterior bulk force (as it was originally done by Fujita and Kato \cite{fujita_kato} for the Navier-Stokes equation, and in \cite{cdgg} for System \eqref{nsce}. The first equation of \eqref{nsce} becomes: 
\begin{equation}
    \partial_t u - \nu \Delta u + \mathbb P \dfrac{e_3 \wedge u}{\varepsilon} = - \mathbb P u\cdot \nabla u
\end{equation}
Duhamel's formula makes the above equation formally equivalent to the following mild formulation:
\begin{equation}
\label{mild}
u(t) = S (t) u_0 - \itg_0^t S(t-t') \mathbb{P}\ u(t') \cdot \nabla u(t') dt' ,
\end{equation}
where $S(t)$ is the semigroup generated by the linear homogeneous equation:
\begin{equation}
    \partial_t v = \nu \Delta v - \mathbb P \dfrac{e_3 \wedge v}{\varepsilon}.
\end{equation}
The mild formulation \eqref{mild} is used as a fixed point equation which is solved using the contraction mapping principle.

Building the contraction map in a suitable space requires estimates for the homogeneous and inhomogeneous terms in the right-hand side of \eqref{mild}. In the case of Navier-Stokes, $S(t)$ is reduced to the heat semigroup and existence can only be proved locally or for small enough initial conditions. However, in our case of interest, the semigroup involves the additional geophysical terms. The central feature of these terms is their dispersive nature, meaning that different frequency modes travel at different speeds, effectively spreading out and regularizing the solution. In \cite{bmn1}, \cite{bmn2}, \cite{cdgg}, the authors made use of the additional regularity provided by rotation to predict the behaviour of the solutions for high rotation speeds. It was proved that for $\varepsilon$ sufficiently small, the solution of \eqref{nsce} converges towards the solution of the classical two-dimensional Navier-Stokes equation, see Theorems 5.6 and 5.7 in \cite{cdgg}: this corresponds to the aforementioned Taylor-Proudman Theorem.

The inviscid case $\nu = 0$ is of particular interest because $S(t)$ is then only generated by the geophysical terms. In lack of the well-known regularizing behaviour of the heat semigroup, capturing all the regularity provided by the geophysical semigroup is crucial. Koh, Lee and Takada \cite{kohleetakada} have derived sharp Strichartz estimates for the semigroup associated to the Coriolis force alone, which are used to prove long-time existence of solutions provided the rotation is sufficiently fast. In \cite{takada}, an analog result to the Taylor-Proudman Theorem was shown for the inviscid Boussinesq system, another geophysically relevant equation that will be studied in this paper. It is made possible by sharp Strichartz estimates for the semigroup associated to fluid stratification.

The Strichartz estimates for the geophysical semigroups are obtained in the literature by relying on a dispersive inequality
\begin{equation}
\label{disp}
    \Vert S(t)u \Vert_{L^\infty_x} \leq C (1+\vert t \vert )^{-\sigma} \Vert u \Vert_{L^1_x},
\end{equation} for some $\sigma>0$, which is then used in a $TT^*$ argument. The dispersive estimates themselves are obtained by the stationary phase method (see Section 5, Theorem 5.3 from \cite{cdgg}) or by Littman's theorem \cite{littman}. It relies on the expression of the semigroup as an oscillatory integral:
\begin{equation*}
    (S(t)u) (x) = \itg_{\R^3} e^{ix\cdot \xi + it p(\xi)}\hat u (\xi) d\xi,
\end{equation*}
where $p(\xi)$ is the dispersion relation associated to the geophysical term. For the Coriolis force (including the Leray projector), $p(\xi)=\frac{\xi_3}{\vert\xi\vert}$. It is homogeneous of degree $0$, which, by the Littlewood-Paley decomposition, allows one to only consider the localized frequency case $r < \vert \xi \vert < R$. Since $p$ is smooth on this annulus, the Littman theorem directly applies and immediately provides \ref{disp} after computing the rank of the Hessian matrix of $p$. However, for the case of the Boussinesq system, the dispersion relation is $p(\xi)=\frac{\vert \xi_h \vert}{\vert\xi\vert}$ where $\xi_h = (\xi_1,\xi_2)$. It is still of degree 0, but not smooth on the domain $r < \vert \xi \vert < R$, so the stationary phase or Littman's theorem do not apply directly. In \cite{lee_takada} the authors manage to circumvent the lack of smoothness with a careful cut-out of the singularities using dyadic decomposition, leading to quite technical computations.

\emph{In this paper, we prove Strichartz inequalities for geophysical fluids without resorting to a dispersive estimate. Our method instead relies on a Fourier restriction estimate on the surface $\{ \xi, p(\xi) \}$, which by duality provides a Strichartz estimate. The proof of the restriction estimate relies on the slicing method \cite{nicola}, and takes place in the framework of measurable functions, for which the singular sets are of zero Lebesgue measure and are thus easily dealt with. This avoids the technicalities involved by possible lack of smoothness in the otherwise very similar dispersion relations. To the authors' knowledge, restriction theory has never been used before to prove Strichartz estimates related to geophysical fluid equations. From this stems our main result, which is the sharp Strichartz estimate for the primitive system. This new method also allows us to recover Strichartz estimates for the Euler equations in a rotating framework (which were obtained in \cite{kohleetakada} following the standard approach described above) and the inviscid stably stratified Boussinesq system (treated in \cite{lee_takada}, \cite{takada}).}

\subsection{Further results}

In this section, we present the two systems obtained when considering either fluid stratification or rotation alone. The associated semigroups have already been studied (\cite{takada}, \cite{kohleetakada}, \cite{lee_takada}) with dispersive methods. The Strichartz estimates below were thus already shown in the aforementioned articles, and what we wish to highlight is how our restriction theory approach provides a new way to recover these estimates.

The stably stratified inviscid \emph{Boussinesq system} formally corresponds to the primitive system without rotation. However, for physical considerations, it is more commonly written as follows: 
\begin{equation}
    \label{bousec1}
    \left\{\begin{array}{rcl}
\partial_t u + (u \cdot \nabla) u  &=& - \nabla p + \theta e_3 \\
\partial_t \theta + (u\cdot\nabla)\theta  &=& -N^2 u_3\\
\nabla \cdot u &=&0\\
u\vert_{t=0} = u^0& \text{and} & \theta\vert_{t=0}=\theta^0,
 \end{array}\right.
\end{equation} 
by introducing $N>0 $ the Brunt-Väisälä frequency (to exactly recover the primitive system, one can consider the change of variable $\tilde \theta = -N^{-1} \theta $). This system describes the evolution of velocity, pressure, and temperature around a motionless equilibrium state, and its obtention from the general Navier Stokes system with an exterior force is presented in Section 4.1. In the absence of rotation, the system still showcases dispersive phenomena, this time related to the semigroup for stratification alone:
\begin{equation}
\label{op_boussi}
    e^{\pm i t  \frac{\vert D_h \vert}{\vert D \vert}}g :(t,x) \longmapsto \itg_{\R^3} e^{ix\cdot \xi \pm it \frac{\vert\xi_h\vert}{\vert\xi\vert}}\hat g (\xi) d\xi,
\end{equation}
where $\xi_h = (\xi_1, \xi_2)$. Restriction theory yields the following Strichartz inequalities for this semigroup which were already obtained in \cite{lee_takada},\cite{takada} using dispersive estimates:
\begin{theorem}
\label{th:bou}
    There exists a universal constant $C>0$ such that the following Strichartz estimate holds for all $g\in \dot{H}^1 (\R^3)$:
\begin{equation}
\label{esti_bou}
    \Vert e^{\pm i t  \frac{\vert D_h \vert}{\vert D \vert}}g \Vert_{L^6_{t} L^6_{x}} \leq C N^{-\frac{1}{6}}  \Vert g \Vert_{\dot H^1}.
\end{equation}
Moreover, for all $f\in L^1_t\dot H^1$, we have
\begin{equation}
\label{duhamel_bou}
    \left\Vert \itg_0^t e^{\pm i (t-t')  \frac{\vert D_h \vert}{\vert D \vert}}f(t') dt' \right\Vert_{L^6_{t} L^6_{x}} \leq C  N^{-\frac{1}{6}} \Vert f \Vert_{L^1_t\dot H^1_x}.
\end{equation} \end{theorem}
The homogeneous estimate is sharp, because the same scaling condition on the exponents \eqref{scaling} holds for the Boussinesq equations. This result, which is the equivalent of Proposition \ref{prop:sharp}, has been shown in \cite{lee_takada}.

Secondly, we consider inviscid fluids subject only to the Coriolis force, which yields the \emph{Euler equations in a rotational framework}:
\begin{equation}
\label{eulertourne}
      {\left\{\begin{array}{rcl}
\partial_t u + u\cdot \nabla u + \dfrac{e_3 \wedge u}{\varepsilon} + \nabla p &=& 0 \\
\nabla \cdot u &=& 0 \\
u(0,\cdot)&=&u_0.
\end{array}\right.}
\end{equation}
The dispersion is provided by the semigroup associated to the fluid rotation alone:
\begin{equation}
\label{op_rot}
    e^{\pm i t  \frac{D_3}{\vert D \vert}}g :(t,x) \longmapsto \itg_{\R^3} e^{ix\cdot \xi \pm it \frac{\xi_3}{\vert\xi\vert}}\hat g (\xi) d\xi.
\end{equation}
Restriction theory provides the following Strichartz estimates:
\begin{theorem}
\label{th:rot}
    There exists a universal constant $C>0$ such that the following Strichartz estimate holds for all $g\in \dot{H}^1 (\R^3)$:
\begin{equation}
\label{esti_eul}
    \Vert e^{\pm i t  \frac{ D_3}{\vert D \vert}}g \Vert_{L^6_{t} L^6_{x}} \leq C \varepsilon^{\frac{1}{6}}  \Vert g \Vert_{\dot H^1}.
\end{equation}
Moreover, for all $f\in L^1_t\dot H^1$, we have
\begin{equation}
\label{duhamel_euler}
    \left\Vert \itg_0^t e^{\pm i (t-t')  \frac{ D_3}{\vert D \vert}}f(t') dt' \right\Vert_{L^6_{t} L^6_{x}} \leq C \varepsilon^{\frac{1}{6}} \Vert f \Vert_{L^1_t\dot H_x^1}.
\end{equation} 
\end{theorem}
The estimate is not sharp in this case, as it was shown in \cite{kohleetakada} that the optimal range of exponents in this case is $\frac{1}{q}+\frac{1}{r}\leq \frac{1}{2}$.
The lack of sharpness in this particular case can be easily tracked down in our method and is discussed in Section 5.3. We will explain how to obtain, still resorting to restriction theory, the improved estimate
    $$ \Vert e^{\pm i t  \frac{ D_3}{\vert D \vert}}g \Vert_{L^6_x L^3_t} \leq C \varepsilon^{\frac 1  3 } \Vert g \Vert_{\dot{H}^1}. $$
in which the exponents are the optimal ones but the time and space norms are reversed in comparison to usual Strichartz estimates.

\begin{Remark} 
Evolution PDEs are usually thought of as being solved for $t \in (0, \infty )$, from an initial condition at time $t=0$. However, the equations at play here are time reversible, meaning they admit solutions for $t\in \R$. Therefore the Lebesgue norms in time in our estimates can be taken over $\R$. This is because of the absence of viscous forces, usually exploited for its additional decay of the solutions through the well-known regularizing properties of the heat semigroup (as it is considered in \cite{charve}, \cite{cdgg}, \cite{kohleetakada} and \cite{lee_takada}). 
\end{Remark}

\begin{Remark}
    We consider inviscid fluids throughout this paper, but the previous systems could all write with additional viscous forces $-\nu \Delta u$ and heat diffusion $-\kappa \Delta \theta$ in the left-hand side of equations. These terms are proportional to the identity in Fourier space (provided $\nu = \kappa$, an assumption often made), which makes them commute with the rotation and stratification. Hence, they can be treated separately and our results can be extended to the viscous case, which is in general easier thanks to the regularizing behavior of the heat semi-group.
\end{Remark}

\subsection{Layout and Notations}

\textbf{Layout. }The article is organized as follows. In Section 2, we illustrate the general method of proof for restriction theorems on surfaces of $\R^d$ by slicing them into surfaces in $\R^{d-1}$, by deriving the cone restriction theorem from the Tomas-Stein theorem applied to its spherical slices, a method initially developed in \cite{nicola}. In Section 3, we prove Theorem \ref{th:pri} on the primitive equations. Theorem \ref{th:bou} on the Boussinesq system is proved in Section 4, and Theorem \ref{th:rot} on the Euler equations in a rotational framework in Section 5. Finally, we prove in Appendix how to obtain the inhomogeneous estimates  (resp. \eqref{duhamel_pri}, \eqref{duhamel_bou} and \eqref{duhamel_euler}) from the homogeneous ones (resp. \eqref{esti_pri}, \eqref{esti_bou} and \eqref{esti_eul}).

\textbf{Notations}. Throughout this article, we will denote by $d\sigma_R$ the surface measure on the sphere $\Spd_R\subset \R^d$ of radius $R$, defined by
$$ d\sigma_R(A) = d\lambda(\{ta: a\in A, 0< t \leq 1\})$$
for any measurable subset $A \subset\Spd_R$, with $d\lambda$ being the Lebesgue measure on $\R^d$.

If $p\in [1,+\infty]$, its Lebesgue dual exponent will be denoted $p'$, such that $\frac 1 p + \frac{1}{p'} = 1$.

The Fourier transform of any function $f$ will be denoted $\hat f$, and the inverse Fourier transform will be written $\mathcal{F}^{-1}$. Moreover, the partial Fourier transform of $f$ in the last coordinate will be denoted $\Grave{f}(x,\rho)$, such that:
$$\Grave{f} (x,\rho) = \int_{\R}e^{-ix_{d+1}\rho}\,f(x,x_{d+1})\,dx_{d+1}.$$

 Finally, in the computations, $C$ will be a positive constant independent of the physical parameters, that may vary from line to line.

\section{On the slicing method}

The techniques developed in the literature to prove restriction estimates often make use of subtle harmonic analysis arguments as the surfaces become increasingly complex. In particular, curvature has been understood to play a key role, as lack of curvature always entails a loss in the exponents of the estimates: it can be  seen through the simple fact that restriction estimates for surfaces with vanishing Gaussian curvature are out of reach of the usual Tomas-Stein theorem. However, in explicit cases, one can sometimes identify a $d$-dimensional surface of interest as a combination of $(d-1)$-dimensional surfaces with better properties. A canonical example is the relation between the cone in $\R^{d+1}$ and the sphere in $\R^d$ \cite{nicola}: horizontally slicing the cone yields spheres of increasing radius, and moreover those spheres encompass all the curvature of the initial cone, so one does not expect additional contribution from the vertical stacking of the spheres. Indeed, the numerology agrees and the sharp results for the cone can easily be retrieved by slicing it in spheres and applying the Tomas Stein theorem. To the authors' knowledge, this slicing method was only formally used to show the connection between the sphere and the cone quite recently \cite{nicola}. Since our framework for the slicing method is a bit different, and because it is central to this paper, we showcase the cone restriction theory and its proof below.

The restriction estimate for the cone is the following:
\begin{theorem}[\cite{stri}]
Let $d\geq 2$ and consider the cone $\cont = \{ (\xi, \vert \xi \vert ),\ \xi \in \R^d \}\subset \R^{d+1}$ endowed with surface measure $d\sigma(\xi,|\xi|) = \frac{d\xi}{|\xi|}$. For $p_0= \frac{2(d+1)}{d+3}$, there exists a constant $C>0$ such that for any function $f\in \s(\R^{d+1})$:
$$ \Vert \hat{f}|_{\cont} \Vert_{L^2(\cont, d\sigma)} \leq C \Vert f \Vert_{L^{p_0}(\R^{d+1})}. $$
\end{theorem}
As stated above, the proof by the slicing method requires an extension of Theorem 1 to spheres of varying radius, which is given by the following lemma. Its proof is a simple scaling argument left to the reader.

\begin{lemma}
Let $\Spd_R$ be the sphere of radius $R>0$ included in $\R^d$, endowed with the surface measure $d\sigma_R$ (not normalized, so that $d\sigma_R(\Spd_R) \propto R^{d-1}$). Then there exists a constant $C>0$ (independent of R) such that for any function $f\in \s(\R^{d})$, one has:
$$ \Vert \hat{f}|_{\Spd_R} \Vert_{L^2(\Spd_R, d\sigma_R)} \leq C R^{\alpha} \Vert f \Vert_{L^{p_0}(\R^{d})},$$
with  $\alpha =  \frac{1}{p_0'} = \frac{d-1}{2(d+1)}$.
\end{lemma}
The slicing method will make use of this lemma to prove Theorem 5. As stated above, this proof is slightly different than what was done in \cite{nicola} because we directly prove the restriction estimate rather than its dual extension estimate.

\begin{proof}[Theorem 5]
Consider $ f\in\s(\R^{d+1}) $. The first step to prove the theorem stated above is to write $\Vert \hat{f}|_{\cont} \Vert_{L^2(\cont, d\sigma)}$ as an integral over spheres, using polar coordinates. We have:
\begin{align*}
    \Vert \hat{f}|_{\cont} \Vert^2_{L^2(\cont, d\sigma)}& = \itg_{\R^d} \big\vert\hat{f}(\xi,|\xi|) \big\vert^2 \frac{d\xi}{|\xi|}\\
    &= \itg_0^{+\infty} \itg_{\Spd_{\rho}} \vert \hat{f}(\xi,\rho) \vert^2 d\sigma_{\rho}(\xi) \frac{d\rho}{\rho},
\end{align*}
with $\Spd_{\rho}$ the sphere of radius $\rho$, and $d\sigma_{\rho}$ the measure on $\Spd_{\rho}$ as in Lemma 1.

Now, we want to use the Tomas-Stein theorem on the sphere of radius $\rho$, that is to say apply Lemma 1. It yields:

$$\itg_0^{+\infty} \itg_{\Spd_{\rho}} \vert \hat{f}(\xi,\rho) \vert^2 d\sigma_{\rho}(\xi) \frac{d\rho}{\rho}
    \leq C \itg_0^{+\infty} \bigg( \itg_{\R^d} \vert \Grave{f}(x,\rho) \vert^{p_0} dx \bigg)^{\frac{2}{p_0}} \frac{d\rho}{\rho^{1-2\alpha}},$$
with $\Grave{f}(x,\rho)$ the partial Fourier transform of $f$ in the last coordinate, that is:
\begin{equation}
    \label{fgrave}
    \Grave{f} (x,\rho) = \int_{\R}e^{-ix_{d+1}\rho}\,f(x,x_{d+1})\,dx_{d+1}.
\end{equation}
At this point, since $p_0 \leq 2$, we can use the Minkowski inequality: 
$$  \itg_0^{+\infty} \bigg( \itg_{\R^d} \vert \Grave{f}(x,\rho) \vert^{p_0} dx \bigg)^{\frac{2}{p_0}} \frac{d\rho}{\rho^{1-2\alpha}} 
\leq C \bigg(\itg_{\R^d} \bigg( \itg_0^{+\infty} \vert \Grave{f}(x,\rho) \vert^2 \frac{d\rho}{\rho^{1-2\alpha}} \bigg)^{\frac{p_0}{2}} dx \bigg)^{\frac{2}{p_0}} .$$
We recognize on the right-hand side a piece of the homogeneous Sobolev norm of $f(x,\cdot )$:
$$ \bigg( \itg_0^{+\infty} \vert \Grave{f}(x,\rho) \vert^2 \frac{d\rho}{\rho^{1-2\alpha}} \bigg)^{\frac{1}{2}} \leq \Vert f(x,\cdot )\Vert_{\dot H^s(\R)},$$
where
$$ \Vert f(x,\cdot )\Vert_{\dot H^s(\R)}=\bigg( \itg_{\R} \vert \Grave{f}(x,\rho) \vert^2 \vert \rho \vert^{2s}d\rho \bigg)^{\frac{1}{2}},$$
with $s = -\frac{1}{2} + \alpha$. Now, by the dual Sobolev embedding, we find exactly $$\Vert f(x,\cdot )\Vert_{\dot H^s(\R)} \leq C \Vert f(x,\cdot ) \Vert_{L^{p_0}(\R)}. $$ 
Indeed, $s = - \frac{1}{d+1}$ and $\frac{1}{2} - \frac{1}{p_0} = \frac{d+1 - (d+3)}{2(d+1)} = s$ since $\alpha = \frac{1}{p_0'}$. Finally, 
\begin{align*}
    \bigg(\itg_{\R^d} \bigg( \itg_0^{+\infty} \vert \Grave{f}(x,\rho) \vert^2 \frac{d\rho}{\rho^{1-2\alpha}} \bigg)^{\frac{p_0}{2}} dx \bigg)^{\frac{2}{p_0}}
    &\leq C \bigg(\itg_{\R^d} \bigg( \itg_{\R} |f(x,y)|^{p_0} dy \bigg)^{\frac{p_0}{p_0}} dx \bigg)^{\frac{2}{p_0}} \\
    &\leq C \Vert f \Vert_{L^{p_0}(\R^{d+1})}^2,
\end{align*}
which concludes the proof.
\end{proof}

The method was showcased here in its simplest form, but can and will be extended to the more complex surfaces that appear in the Fourier analysis of geophysical equations. The Strichartz estimates of interest in this article will be derived using this method. Compared to the classical dispersive approach described above, it is shorter and better suited to the Lebesgue setting of measurable functions than the Littman theorem that requires the careful removal of non-smooth singularities of the surface. The cost of this method is that we obtain a smaller range of exponents after interpolation of the estimate obtained through restriction with energy conservation.

As can be seen in the proof above, the Sobolev weights that appear after applying the Tomas-Stein theorem to the slice are a testimony of the sharpness of the result: it exactly yields the corresponding $1$-dimensional Sobolev embedding (that is to say, giving the desired $L^p$ norm). In the geophysical cases, several Sobolev weights centered on different points of the surface appear, that sometimes do not match and must be bluntly bounded. Since every other inequality used in the proofs is sharp, it is easy to relate the sharpness of the final result to the loss in the Sobolev weights or not.

\section{The Primitive System}

\subsection{Linear analysis}

The primitive system \eqref{pe}, used for instance in climate models to describe the ocean motion, describes the effects of the rotation of the Earth, the buoyancy, and the stratification on the fluid. These effects appear in the equations through linear terms that can be combined into a skew-symmetric matrix operator:
letting $U=(u_1,u_2,u_3,\theta)^T$, we can rewrite the first two equations of \eqref{pe} as follows:
\begin{equation}
\label{eq:primatrix}
\partial_t U + (u \cdot \nabla) U + \dfrac{1}{\varepsilon} A U = - \dfrac{1}{\varepsilon} (\nabla p,0)
\end{equation} 
where $A$ is the skew-symmetric matrix
$$A=\left(\begin{array}{cccc}
    0 & -1 & 0 & 0  \\
    1 & 0 & 0 & 0 \\
    0 & 0 & 0 & F^{-1} \\
    0 & 0 & -F^{-1} & 0 \\
\end{array}\right).$$

The operator $A$ disappears in energy estimates, however it still exhibits a dispersive nature that can be encapsulated in Strichartz estimates. We begin with a linear analysis of this operator to make the connection with the semigroup in Theorem \ref{th:pri} apparent.

We denote by $\mathbb{P}$ the Leray projector onto divergence-free vector fields. We allow ourselves to also use $\mathbb{P}$ for its extension to 4-dimensional vectors, such that $\mathbb{P} U = (\mathbb{P}u, \theta)^T$. Applying the Leray projector  on \eqref{eq:primatrix} to eliminate the pressure $p$ yields:
\begin{equation}
\partial_t U + \mathbb{P}(u \cdot \nabla) U + \dfrac{1}{\varepsilon}\mathbb{P} A U = 0.
\end{equation} 
As explained in Section 1.5, the treatment of the full non-linear problem first relies on the study of the homogeneous linear equation:
\begin{equation}
\partial_t U = -  \dfrac{1}{\varepsilon}\mathbb{P} A U.
\end{equation}
and the semigroup $S(t)$ it generates. In Fourier space, by plugging in the expression of the Leray projector, a direct matrix product gives:
$$-\widehat{\mathbb{P} A U} = \mathcal{A}\hat{U} $$
where$$\mathcal{A}=\left(\begin{array}{cccc}
\dfrac{\xi_1 \xi_2}{ \vert \xi \vert^2}     &  \dfrac{\xi_2^2 + \xi_3^2}{\vert \xi \vert^2} & 0 & \dfrac{\xi_1 \xi_3}{ F \vert \xi \vert^2}\\
-\dfrac{\xi_1^2 + \xi_3^2}{\vert \xi \vert^2}     & -\dfrac{\xi_1 \xi_2}{ \vert \xi \vert^2} & 0 & \dfrac{\xi_2 \xi_3}{ F \vert \xi \vert^2}\\
\dfrac{\xi_2 \xi_3}{ \vert \xi \vert^2} & -\dfrac{\xi_1 \xi_3}{ \vert \xi \vert^2}
& 0 & -\dfrac{\xi_1^2 + \xi_2^2}{F \vert \xi \vert^2}\\
0 & 0 & \dfrac{1}{F} & 0 \end{array}\right) .$$

We only need to study $\mathbb{P} A$ on the subspace of divergence-free vector fields, that is, study $\mathcal{A}$ on the subspace of vectors orthogonal to $(\xi_1, \xi_2, \xi_3,0)^T$.  A computation of the eigenelements of $\mathcal{A}$ in this subspace gives the following three eigenvalues:
$$\left\{ 0, \pm i p_F(\xi) \right\}$$
where we recall that $$p_F(\xi) = \dfrac{ \vert \xi \vert_F}{F \vert \xi \vert}\ , \ \ \ \ \ \vert \xi \vert_F=\vert (\xi_1, \xi_2, F \xi_3)\vert.$$

The corresponding eigenvectors $a_0$, $a_+$, $a_-$ are orthogonal, which can be deduced from their explicit expression (that has otherwise no use, so we omit to give it here). This allows us to decompose the semigroup $S(t)$ using its action on the different eigenvectors in the following manner:
\begin{align*}
    S(t) U_0 =\  & e^{it \frac 1 \varepsilon p_F(D) } \mathcal{F}^{-1}\bigg( ( \hat U_0  \cdot a_+ ) a_+ \bigg) \\
    +&\  e^{-it\frac 1 \varepsilon p_F(D) } \mathcal{F}^{-1}\bigg( ( \hat U_0  \cdot a_- ) a_- \bigg) \\
    +&\    \mathcal{F}^{-1}\bigg( ( \hat U_0  \cdot a_0 ) a_0 \bigg) .
\end{align*}
The component along $a_0$ of $U_0$, which is trivially propagated, is called the \textit{quasigeostrophic part}. The other two components are referred to as the \textit{oscillatory parts} and are the ones subject to dispersion.
This is why we are interested in the semigroup $e^{\pm i t \frac 1 \varepsilon p_F(D)}$, which propagates an initial data $g$ as:
$$ e^{\pm i t \frac 1 \varepsilon p_F(D)}g : (t,x) \longmapsto \itg_{\R^3} e^{ix\cdot \xi \pm i t p_F(\xi)}\hat g (\xi) d\xi.$$
Now that the connection with the primitive system has been made explicit, we prove the Strichartz estimate for this semigroup.

\subsection{Proof of Theorem \ref{th:pri}}
We can set $\varepsilon=1$, since the estimates on the semigroup generated by $\frac{1}{\varepsilon} \mathbb{P} A$ for any $\varepsilon$ can be derived by rescaling time by $\varepsilon$ (this rescaling explains the dependence in $\varepsilon$ of the constant in Theorem \ref{th:pri}). Our first goal is to reduce the proof of the Strichartz estimate to a Fourier restriction estimate on the surface
$$\mathcal{M}_{F}=\left\{ \zeta = \left(\xi,p_F(\xi)\right), \xi \in \R^{3}\right\}\subset \R^4.$$
We endow $\mathcal{M}_{F}$ with the surface measure $d\sigma(\zeta) = \frac{d\xi}{|\xi|^2}$. Denoting $\tilde g (\xi, p_F(\xi)) = \hat g(\xi)$ the projection onto the surface of the Fourier transform of $g \in \s(\R^3)$, the map $g \mapsto \tilde g$ is an isometry between spaces $\dot H^1(\R^3) \rightarrow L^2(\mathcal{M}_F, d\sigma)$. Using this isometry, the Strichartz estimate we seek thus rewrites as the continuity of the following extension operator:
\begin{equation}
    L^2(\mathcal{M}_F, d\sigma) \ni \tilde g \longmapsto \left( (t,x) \mapsto \itg_{\mathcal{M}_F} e^{i(x,t)\cdot \zeta} \tilde g (\zeta) d\sigma (\zeta) \right) \in L^{6}(\R^4).
\end{equation}
The continuity of this operator is equivalent to the continuity of the dual restriction operator $f \mapsto \hat{f}\vert_\mathcal{M}$ between $L^\frac{6}{5}(\R^4) \rightarrow L^2(\mathcal{M}, d\sigma)$, moreover the two share the same operator norm. The Strichartz estimate we seek hence results from the following restriction estimate:

\begin{lemma}
Consider the surface $\mathcal{M}_{F}$ endowed with the measure $d\sigma(\xi,p_F(\xi)) = \frac{d\xi}{|\xi|^2}$ and with $F\in (1,\infty)$. There exists a constant $C>0$ such that for any $f\in\s(\R^4)$,
$$\Vert \hat f\vert_{\M_F}  \Vert_{L^2(\M_F,d\sigma)} \leq C \frac{F^{\frac 1 2}}{( F-1 )^{\frac 1 6}} \Vert f \Vert_{L^{\frac{6}{5}}(\R^4)}. $$
\end{lemma}

We will adapt the slicing method of the previous section: indeed, notice that the level sets of $\M_F$ are cones of varying opening angles. Our goal is thus to write the $L^2$ norm of $f\vert_\M$ as an integral over all level sets and apply a suitable restriction theorem at each level.
We thus need to extend Theorem 5 into a corollary for cones of arbitrary angle:
\begin{Corollary}
Let $\cont_{\rho}:= \{ (\xi, \rho \vert \xi \vert),\ \xi\in\R^2 \}$ be a cone in  $\R^3$, with surface measure $d\sigma_\rho =\frac{d\xi}{\vert \xi \vert}$. Then, for any function $f\in \s(\R^{3})$, one has:
$$ \Vert \hat{f}|_{\cont_{\rho}} \Vert_{L^2(\cont_{\rho}, d\sigma_\rho)} \leq C \vert \rho \vert^{-\frac{1}{6}}\vert\  \Vert f \Vert_{L^{\frac 6 5}(\R^{3})}. $$
\end{Corollary}

This is directly deduced from the case $\rho=1$ by rescaling correctly $f$ in the third variable, with $\frac 6 5$ simply being the $p_0$ exponent in Theorem 2 for cones in $\R^3$. We use this to prove Lemma 2.

\begin{proof}[Lemma 2]
We proceed by successively bounding the left-hand side,
$$\Vert \hat f\vert_{\M_F}  \Vert^2_{L^2(\M,d\sigma)} = \itg_{\R^3} \left\vert \hat f \left(\xi, p_F(\xi)\right) \right\vert^2 \frac{d\xi}{\vert \xi\vert^2}\ \cdotp$$
We can split the integral defining $\left\Vert \hat f\vert_{\mathcal{R}} \right\Vert_{L^2(\mathcal{R}, d\mu)}^2$ over the $\xi_3>0$ and $\xi_3<0$ domains. We can restrict our computations to one half of this surface, say $\xi_3>0$, because the other is identical. It does not change the result (up to doubling a constant) by symmetry of $\mathcal{M}_F$.

Let $\xi_h = (\xi_1,\xi_2)$. The first step is to perform the change of variable from $\xi_3$ to $$\mu = p_F(\xi) =  \dfrac{\vert \xi \vert_{F}}{F\vert\xi\vert}$$ in order to slice the surface in level sets of constant $\mu$. We use for this the following identities:
\begin{align*}
&\xi_3 = \sqrt{ \frac{\mu^2-F^{-2}}{1-\mu^2} }\ \ \vert\xi_h\vert, \\
&d\xi_3 = \mu \frac{1-F^{-2}}{(1-\mu^2)^{\frac 3 2}(\mu^2-F^{-2})^{\frac 1 2 }}\ \ \vert\xi_h\vert\ d\mu, \\
\text{and }\ &
\vert\xi\vert^2 = \vert\xi_h\vert^2\bigg(\frac{1-F^{-2}}{1-\mu^2}\bigg).    
\end{align*}

Plugging in these three relations yields
\begin{align*}
 \itg_{\R^2_{\xi_h}\times\R^+_{\xi_3}} & \left\vert \hat f \left(\xi, \dfrac{\vert \xi \vert_{F}}{F\vert\xi\vert}\right) \right\vert^2 \frac{d\xi}{\vert \xi\vert^2}  \\
 &=\itg_{F^{-1}}^1 \itg_{\R^2_{\xi_h}} \left\vert \hat f \bigg(\xi_h, \sqrt{ \frac{\mu^2-F^{-2}}{1-\mu^2} }\ \ \vert\xi_h\vert , \mu \bigg) 
 \right\vert ^2  \frac{d\xi_h}{\vert \xi_h\vert} \ \mu (1-\mu^2)^{-\frac 1 2}(\mu^2-F^{-2})^{-\frac 1 2 }\ \  d\mu .
\end{align*}

For any given $\mu$ the integral in $d\xi_h$ is an $L^2$ norm of a Fourier transform on a cone, so we can apply Corollary 1 with $\rho = \sqrt{ \frac{\mu^2-F^{-2}}{1-\mu^2} }$, letting $p=\frac 6 5$. This transforms back to physical space the first three coordinates of $\hat{f}$ and we obtain a bound involving an $L^p$ norm of $\Grave{f}$ (the partial Fourier transform of $f$ in the fourth coordinate):
\begin{align*}
\itg_{\R^2_{\xi_h}} &\left\vert \hat f \bigg(\xi_h, \rho \vert\xi_h\vert , \mu \bigg) \right\vert^2  \frac{d\xi_h}{\vert \xi_h\vert} \leq  C \rho^{-\frac{1}{3}} \bigg(\itg_{\R^3_x} \vert \Grave{f}(x,\mu) \vert^p dx \bigg)^{\frac 2 p}.
\end{align*}
Using this estimate as well as keeping track of the $\mu$ terms in $\rho$, one gets
\begin{align*}
    \itg_{F^{-1}}^1 \itg_{\R^2_{\xi_h}} &\left\vert \hat f \bigg(\xi_h, \rho \vert\xi_h\vert , \mu \bigg) \right\vert^2  \frac{d\xi_h}{\vert \xi_h\vert} \ \mu (1-\mu^2)^{-\frac 1 2}(\mu^2-F^{-2})^{-\frac 1 2 }\ \  d\mu\\
    &\leq C \itg_{F^{-1}}^1 \bigg(\itg_{\R^3_x} \vert \Grave{f}(x,\mu) \vert^p dx \bigg)^{\frac 2 p} \mu (1-\mu^2)^{-\frac 1 3}(\mu^2-F^{-2})^{-\frac 2 3}d\mu.
\end{align*}
The Minkowski inequality turns the $L^2_{\mu}L^p_x$ norm into a $L^p_x L^2_{\mu}$ norm:
\begin{align}
\label{eq:19}
     \itg_{F^{-1}}^1 \bigg(\itg_{\R^3_x} \vert \Grave{f}(x,\mu) \vert^p dx \bigg)^{\frac 2 p}& \mu (1-\mu^2)^{-\frac 1 3}(\mu^2-F^{-2})^{-\frac 2 3}d\mu\\ 
     \nonumber
     \leq \Bigg( \itg_{\R^3_x} \bigg( &\itg_{F^{-1}}^1 \vert \Grave{f}(x,\mu) \vert^2 \mu (1-\mu^2)^{-\frac 1 3}(\mu^2-F^{-2})^{-\frac 2 3}d\mu \bigg)^{\frac p 2 } dx \Bigg)^{\frac 2 p}.
\end{align}
To conclude, it is enough to prove that $$\itg_{F^{-1}}^1 \vert \Grave{f}(x,\mu) \vert^2 \mu (1-\mu^2)^{-\frac 1 3}(\mu^2-F^{-2})^{-\frac 2 3}d\mu \leq C_F \Vert f(x,\cdot) \Vert_{L^p(\R)}^2.$$

To do so, we shall split the above integral around the singularities at $F^{-1}$ and $1$, and make sure we can use the Sobolev embedding $L^{\frac 6 5}(\R) \hookrightarrow \dot H^{-\frac 1 3}(\R)$ for each integral piece. We thus write
$$\itg_{F^{-1}}^1 \vert \Grave{f}(x,\mu) \vert^2 \mu (1-\mu^2)^{-\frac 1 3}(\mu^2-F^{-2})^{-\frac 2 3}d\mu  = \itg_{F^{-1}}^{ \frac{F^{-1}+1}{2} } ... + \itg_{ \frac{F^{-1}+1}{2}  }^1 ...$$
For the first integral piece (with $\mu \in [F^{-1}, \frac{F^{-1}+1}{2} ]$), the weight in $\mu$ satisfies: $$\mu (1-\mu^2)^{-\frac 1 3}(\mu^2-F^{-2})^{-\frac 2 3} \leq C \frac{F}{(F-1)^{\frac 1 3}} (\mu - F^{-1})^{-\frac 2 3},$$ and this is exactly the exponent $2s$ we seek for a Sobolev weight with $s = -\frac 1 3$. This means we can bound the first integral piece by $ \Vert \Grave{f}(x,\cdot) \Vert^2_{\dot H^{- \frac 1 3}(\R)}$. Remark that we could not do better in view of the weight in $\mu$: the optimal exponent naturally appears in the computation. This optimality is a testimony of the sharpness of the result, which will be formally proved with Proposition 1.

For the second integral piece (with $\mu \in [\frac{F^{-1}+1}{2},1]$), the singularity at $1$ exhibits an exponent $2s=-\frac 1 3$ rather than $-\frac 2 3$. We have to bluntly bound $\vert 1-\mu \vert^{-\frac 1 3}\leq \vert 1-\mu \vert^{-\frac 2 3}$ to recover the desired Sobolev weight. This might seem brutal but cannot be avoided since the singularity at $F^{-1}$ already prevented to have a better  (greater) exponent than $-\frac 2 3$. We make use of this room to have a better dependency in $F$, an we shall use the inequality $(\mu-F^{-1})^{-\frac 2 3} \leq \bigg(\frac{F}{F-1}\bigg)^{\frac 1 3} (1-\mu)^{-\frac 1 3}$. It yields
$$\mu (1-\mu^2)^{-\frac 1 3}(\mu^2-F^{-2})^{-\frac 2 3} \leq C \bigg(\frac{F}{F-1}\bigg)^{\frac 1 3} (1-\mu)^{-\frac 2 3}.$$
Once again it is the desired Sobolev weight and we can bound this second integral piece by the Sobolev norm $ \Vert \Grave{f}(x,\cdot) \Vert^2_{\dot H^{- \frac 1 3}(\R)}$. The aforementioned Sobolev embedding $L^{\frac 6 5}(\R) \hookrightarrow \dot H^{-\frac 1 3}(\R)$ used in the last coordinate then gives
$$ \itg_{F^{-1}}^1 \vert \Grave{f}(x,\mu) \vert^2 \mu (1-\mu^2)^{-\frac 1 3}(\mu^2-F^{-2})^{-\frac 2 3}d\mu \leq C \frac{F}{(F-1)^{\frac 1 3}}\Vert f(x,\cdot) \Vert_{L^p(\R)}^2.$$
Plugging this in \eqref{eq:19} concludes the proof of both Lemma 2 and the homogeneous estimate \eqref{esti_pri} of Theorem \ref{th:pri}. 

In order to have the inhomogeneous estimate \eqref{duhamel_pri}, we simply apply Lemma 5 stated in the Appendix, with $p=q=6$, $s=1$, and $C_0= C \varepsilon^{\frac{1}{6}} \frac{F^{\frac 1 2}}{( F-1 )^{\frac 1 6}} $. This is immediate, as the semigroup considered satisfies all the hypothesis, concluding the proof of Theorem \ref{th:pri}.
\end{proof}

\begin{Remark}
    The fact that the estimates blow up as $F\rightarrow 1$ is very intuitive : the surface then becomes the hyperplane $\{\xi_4 = 1\}$, for which no restriction estimate holds, as seen in the introduction. We refer to \cite{cheminpenal} for a study of the primitive equations in the case $F=1$ with viscosity.
\end{Remark}

\subsection{Sharpness: proof of Proposition 1}

We now show the sharpness of Theorem \ref{th:pri}
by proving Proposition 1.
\begin{proof}[Proposition \ref{prop:sharp}]
    The proof is based on \cite{lee_takada} where the Boussinesq system was treated. Let $p,\ q,\geq 0$, and $s \in \R$, and assume the Strichartz estimate \eqref{sharppri} holds. We can assume without loss of generality that $\varepsilon = 1$.
    Let $R>0$ with no assumption on its size, and $0<\delta <1$. 

Consider a set of frequencies with small $\xi_3$:
$$\Lambda_{\delta,R} = \left\{ \xi \in \R^3\ \vert \ \vert\xi_1\vert,\vert \xi_2\vert \in \bigg(\frac{ R}{2}, R\bigg) \ \ \text{and} \ \ \vert\xi_3\vert \in \bigg(\frac{\delta R}{2} , \delta R\bigg) \right\}. $$ Define the function $g_{{\delta,R}}$ such that its Fourier transform is the following characteristic function:
$$\hat g_{\delta,R} ={\left \{\begin{array}{lrcl}
1 & \text{if} & \xi \in \Lambda_{\delta,R} \\
0 & \text{if} & \xi \not\in \Lambda_{\delta,R} , \end{array}\right.}$$
Fix a $N$ that will be taken larger enough later, and consider the following set in physical space:
\begin{align*}
 L_{\delta, R, N}= \bigg\{ (t,x) \in \R\times \R^3, \ \vert x_1 \vert,\ \vert x_2 \vert & \in (\frac{1}{2N  R}, \frac{1}{N R}),\\
 \vert x_3 \vert & \in (\frac{1}{2N \delta R}, \frac{1}{N \delta R}), \ \vert t \vert \in (\frac{1}{2N \delta^2 }, \frac{1}{N \delta^2 }) \bigg\}.
\end{align*} 

We evaluate the norm of $g_{\delta ,R}$ and of $e^{\pm i t p_F(D)}g$. The initial data can be evaluated as so:
\begin{align*}
    \Vert g_{\delta,R} \Vert_{\dot H^s}^2 &=\itg_{\R^3} \vert \hat g_{\delta,R} (\xi) \vert^2 \ \vert \xi \vert^{2s} d\xi \\
    &=\itg_{\Lambda_{\delta,R}} \vert\xi\vert^{2s} d\xi \\
    & \simeq  R^3 \times \delta \times R^{2s}.
\end{align*}
where $a\simeq b$ means that there exists some universal constant $C>0$ such that $C^{-1} b \leq a \leq C b$.

On the other hand, notice that for any $(t,x) \in \R\times \R^3$, we can change the phase in the semigroup by any amount independent of $\xi$. We rotate it by $-\frac{t}{F}$, because the singularity at $F^{-1}$ in the proof of Theorem \ref{th:pri} was the one that saturated the inequalities. Hence we have:
$$ \vert e^{\pm i t p_F(D)}g_{\delta,R} (t,x) \vert =   \left\vert \itg_{\R^3} e^{ix\cdot \xi + it p_F(\xi)}\hat g_{\delta,R} (\xi) d\xi \right\vert = \left\vert \itg_{\Lambda_{\delta,R}} e^{ix\cdot \xi + it (p_F(\xi)-\frac{1}{F})} d\xi \right\vert , $$
where we have also used the definition of $g_{\delta,R}$.

Now, we use the following expansion of the phase:
\begin{align*}
    p_F(\xi)-\dfrac{1}{F} &= \dfrac{\vert \xi \vert_F}{F \vert \xi \vert}-\dfrac{1}{F}\\ &= \frac{\vert \xi \vert_F - \vert \xi \vert}{F \vert \xi \vert} \\
    &= \frac{\vert \xi \vert_F^2 - \vert \xi \vert^2}{F \vert \xi \vert(\vert \xi_F \vert + \vert \xi \vert)} \\
    &\simeq \frac{(F^2-1)\xi_3^2}{F\vert\xi\vert^2} \\
    &\simeq \delta^2,
\end{align*}
with the last inequality holding for $\xi \in \Lambda_{\delta,R}$ (and the constant of the $\simeq$ depending on $F$). Therefore, for $(t,x)\in L_{\delta,R,N}$, by construction we have 
$$\vert ix\cdot \xi + it (p_F(\xi)-\frac{1}{F}) \vert \lesssim \frac 1 N,$$
so that $e^{ix\cdot \xi + it (p_F(\xi)-\frac{1}{F})}\simeq 1$
which in turn yields for $N$ large enough (but independent of $\delta$ and $R$),
$$  \itg_{\R^3} e^{ix\cdot \xi + it (p_F(\xi)-\frac{1}{F})}\hat g_{\delta,R} (\xi) d\xi \simeq \vert \Lambda_{\delta,R} \vert \simeq R^3\delta .$$
Therefore, the Lebesgue space-time norm is bounded by below by considering only $L_{\delta,R,N}$:
\begin{align*}
    \Vert e^{\pm i t p_F(D)}g_{\delta,R} \Vert_{L^p_tL^q_x(\R^4)} &\geq \Vert e^{\pm  i t p_F(D)}g_{\delta,R} \Vert_{L^p_tL^q_x(L_{\delta,R,N})} \\
    &\geq C_N R^3 \delta \times R^{-\frac 3 q} \delta^{-\frac 1 q} \times \delta^{-\frac 2 p}.
\end{align*}

Taking $g=g_{\delta,R}$ in the Strichartz estimate \ref{sharppri}, and plugging in the estimates of the norm, we find:
\begin{equation}
    R^3 \delta \times R^{-\frac 3 q} \delta^{-\frac 1 q} \times \delta^{-\frac 2 p} 
    \leq C_N  \Vert e^{\pm i t p_F(D)}g_{\delta,R} \Vert_{L^p_t L^q_x}
    \leq C \Vert g_{\delta, R} \Vert_{\dot H^s} 
    \leq R^{\frac 3 2 -s} \delta^{\frac 1 2}.
\end{equation}
As $R$ can be both very large and very small, it imposes the well-known equality of Sobolev embeddings in 3 dimensions $s = 3(\frac 1 2 - \frac 1 q)$. Letting $\delta \rightarrow 0$ implies $\frac 1 2 - \frac 1 q - \frac 2 p \geq 0$, which is the other desired inequality.

\end{proof}

\begin{Remark} 
We mentioned in the introduction that we cannot have a Lebesgue regularity for the initial data if $p,q,r >2$.

Indeed, assume an inequality of the type \begin{equation}
    \Vert e^{\pm \frac 1 \varepsilon i t p_F(D)}g \Vert_{L^p_t L^q_x}\leq C \Vert g \Vert_{L^r_x}
\end{equation} 
holds for all $g\in L^r_x$. Then, using a Sobolev embedding, this would imply inequality \ref{sharppri} with $s=3(\frac 1 2 - \frac 1 r )$. 

However, we have seen that if this is true, the only admissible Sobolev exponent is the $s$ associated to $q$. Therefore, $r=q$, and the inequality writes 
$$ \Vert e^{\pm \frac 1 \varepsilon i t p_F(D)}g \Vert_{L^p_t L^q_x}\leq C \Vert g \Vert_{L^q_x}.$$
Denote by $f(t,x):=  e^{\pm \frac 1 \varepsilon i t p_F(D)}g$ and assume $g\in \s(\R^3)$ to avoid any regularity consideration. The above inequality implies that $f(t, \cdot) \in L^q_x$ for all $t$, meaning we can plug instead of $g$ any $f(t, \cdot)$ on the right-hand side, and this does not impact the left-hand side as $t\in \R$ (by time-reversibility of the semigroup, see Remark 1). Yet, $\Vert f(t, \cdot)  \Vert_{L^q_x}$ must take arbitrarily small values to ensure the time integrability, so we can force $\Vert f \Vert_{L^p_t L^q_x}=0$, which would imply $g=0$.

\end{Remark}

\section{Stably stratified Boussinesq system}
\subsection{Derivation of the system and linear analysis}

This section focuses on the stably stratified Boussinesq equations and the proof of Theorem \ref{th:bou}. The general Boussinesq system governs the motion and convection of a 3D incompressible fluid — here considered inviscid, see Remark 4 — and writes:
\begin{equation}
    \label{boussinesq}
    \left\{\begin{array}{rcl}
\partial_t u + (u \cdot \nabla) u  &=& - \nabla p + T e_3 \\
\partial_t T + (u\cdot\nabla)T & = & 0 \\
\nabla \cdot u &=&0\\
u\vert_{t=0} = u^0& \text{and} & T\vert_{t=0}=T^0.
 \end{array}\right.
\end{equation} 
It is well-known that this system can describe classical phenomena such as Rayleigh-Bénard convection, which essentially corresponds to the fact that cold air or cold water tends to go below hot air or hot water. The motion of the oceans is also closely linked to this system, since the temperature is non-homogeneous on the Earth, and convection linked to temperature differences are some of the key characteristics of both oceans and atmosphere.

The above system admits a motionless equilibrium solution, with stratified temperature and pressure:
\begin{equation}
    \label{boussinesq_eq}
    \left\{\begin{array}{rcl}
u_{eq} &=&0\\
T_{eq}&=& ax_3 +b \\
p_{eq}&=& \frac a 2 x_3^2 + bx_3 +c, \\
 \end{array}\right.
\end{equation} 
where $a,b,c\in \R$. If the fluid is stable, we shall have $\frac{dT_{eq}}{dx_3} = a := N^2 >0$, an hypothesis that we make in the sequel. It simply means that the hotter layers of fluid are above the denser cold layers.
\begin{Remark}
    Assuming hydrostatic balance ($u=0$) in \eqref{boussinesq} reduces the equations to $\nabla p = T e_3$, which imposes that the pressure $p$ and hence the temperature $T$ depend only on $x_3$, but does not further prescribe their values. The temperature $T$ is chosen as an affine function because if one had taken into account heat diffusivity, then a steady-state equilibrium should additionally have satisfied $\Delta T = 0$. This choice remains the common explicit solution around which perturbations are considered even in the inviscid case (see \cite{widmayer}, whose approach is also reproduced in \cite{lee_takada}, \cite{takada}).
\end{Remark}

Setting $\theta = T-T_{eq}$ and $q=p-p_{eq}$, we can write the equations for perturbations around this equilibrium state as follows:
\begin{equation}
\label{bou}
    \left\{\begin{array}{rcl}
\partial_t u + (u \cdot \nabla) u  &=& - \nabla q + \theta e_3 \\
\partial_t \theta + (u\cdot\nabla)\theta  &=& -N^2 u_3\\
\nabla \cdot u &=&0\\
u\vert_{t=0} = u^0& \text{and} & \theta\vert_{t=0}=\theta^0.
 \end{array}\right.
\end{equation} 
This is how we obtain the inviscid Boussinesq system for a stably stratified fluid. We can merge velocity and temperature in a single 4-dimensional vector $U=(u, \frac{\theta}{N})^T$. Then, setting 
$$  J = \begin{pmatrix} 0&0&0&0\\0&0&0&0 \\0&0&0&-1\\0&0&1&0 \end{pmatrix}$$
the first two equations of System \eqref{bou} write 
\begin{equation}
    \label{bou2}
\partial_t U + (u \cdot \nabla) U + NJU = - (\nabla p,0).
\end{equation} 
Once again, the geophysical effects are concentrated in the skew-symmetric term $NJU$, which vanishes in energy estimates but exhibits a dispersive nature  that can be encapsulated in a Strichartz estimate. Applying the Leray projector $\mathbb P$ onto divergence-free fields gets rid of the pressure, and the linear equation writes
\begin{equation}
    \label{boulin}
    \partial_t U + N \mathbb P J  \mathbb P U = 0
\end{equation}
Here we have used that $ \mathbb P J U =  \mathbb P J \mathbb P U$ since $u$ is already divergence-free, which preserves the skew-symmetry of the operator.

We now perform the linear analysis of the operator $\mathbb P J \mathbb P$. The process is very similar to the previous section. We study its Fourier expression $\widehat{\mathbb P J \mathbb P}$ on the subspace of vectors orthogonal to $(\xi_1, \xi_2, \xi_3,0)^T$ because of the divergence-free condition on $u$. A direct computation from the matrix expression shows that the eigenfrequencies of $\widehat{\mathbb P J \mathbb P}$ on this subspace are $$\left\{ \pm i \frac{\vert\xi_h \vert}{\vert\xi\vert},0\right\}.$$ The associated eigenvectors $a_+, a_-, a_0$ are explicitly computed in \cite{lee_takada} but their exact expression is not of much use in our analysis. Their orthogonality allows to write the semigroup $S(t)$ generated by $N\mathbb P J \mathbb P$ as:
\begin{align*}
    S(t) U_0 =\  & e^{itN\frac{\vert D_h\vert}{\vert D \vert} } \mathcal{F}^{-1}\bigg( ( \hat U_0 \cdot a_+ ) a_+ \bigg) \\
    +&\  e^{-itN\frac{\vert D_h\vert}{\vert D \vert} } \mathcal{F}^{-1}\bigg( ( \hat U_0 \cdot a_- ) a_- \bigg) \\
    +&\    \mathcal{F}^{-1}\bigg( ( \hat U_0 \cdot a_0 ) a_0 \bigg) .
\end{align*}
We find once again the existence of an \textit{oscillatory} component (along $a_\pm$) and a \textit{quasigeostrophic} component (along $a_0$). We now derive the Strichartz estimate for the semigroup $e^{\pm itN\frac{\vert D_h\vert}{\vert D \vert} }$, that is the proof of Theorem \ref{th:bou}.

\subsection{Proof of Theorem \ref{th:bou}}

 We once again set $N=1$ since the general case is directly retrieved by scaling the time variable. The proof of the inhomogeneous estimate is again a direct consequence of Lemma 5 once the homogeneous estimate is proven, so we will not give further details. The exact same reasoning as in the beginning of Section 3.2 reduces the proof of the homogeneous Strichartz estimate for $e^{\pm it\frac{\vert D_h\vert}{\vert D \vert} }$ to the proof of a restriction estimate, this time on the surface $\mathcal{R} \in \R^4$ defined by $$\mathcal{R} = \{(\xi_h,\xi_3,\frac{\vert \xi_h \vert}{\vert \xi \vert}) ,\ \xi = (\xi_h,\xi_3)\in\R^3 \}.$$
Theorem \ref{th:bou} is hence a consequence of the following continuity estimate for the operator $f \mapsto \hat f\vert_\mathcal{R}$:
\begin{lemma}
\label{lem:restribou}
There exists a constant $C>0$ such that for any $f\in\s(\R^4)$, one has
$$ \Vert \hat f\vert_{\mathcal R}  \Vert_{L^2(\M,d\mu)} \leq C \Vert f \Vert_{L^{\frac{6}{5}}(\R^4)}. $$
\end{lemma}
\begin{proof}
 We rely on the slicing method, similarly to what was done in Section 3. We want to slice the integral defining the $L^2$ norm on $\mathcal R$ into partial integrals on the level sets of $\mathcal R$, which are once again cones of various angles. We can restrict our computations to one half of this surface, say $\xi_3>0$. The other half, $\xi_3<0$, is identical. Therefore, we set $\alpha = \frac{\vert \xi_h \vert}{\vert\xi\vert}$ and change the $\xi_3$ variable, with the map $\xi_3 \in (0,\infty)  \mapsto \alpha \in (0,1)$ being one-to-one. Straightforward computations give the identities $$\xi_3 = \frac{\sqrt{1-\alpha^2}}{\alpha} \vert \xi_h \vert \ \text{,}\ \ \ \ \  \ \ \ \ \ d\xi_3 = -\frac{\vert\xi_h\vert}{\alpha^3 \frac{\sqrt{1-\alpha^2}}{\alpha}} d\alpha\ \ \ \ \ \ \text{and}\ \ \ \ \ \   \vert \xi\vert^2 = \frac{\vert\xi_h\vert^2}{\alpha^2} \cdotp$$

Noticing that $0<\alpha <1$, the change of variable yields
\begin{equation*}
\label{boussi1}
\left\Vert \hat f\vert_{\mathcal{R}} \right\Vert_{L^2(\mathcal{R}, d\mu)}^2 = \itg_0^1 \itg_{\R^2_{\xi_h}} \left\vert \hat f \bigg( \xi_h,\frac{\sqrt{1-\alpha^2}}{\alpha} \vert \xi_h \vert,\alpha \bigg) \right\vert^2 \frac{d\xi_h}{\vert\xi_h\vert} \frac{d\alpha}{\sqrt{1-\alpha^2} }.
\end{equation*}
Now, we notice that for a given level $\alpha$, the integral over $\xi_h$ is the integral on a cone of angle $\frac{\sqrt{1-\alpha^2}}{\alpha}$. Therefore, we can apply Corollary 1 with $\rho = \frac{\sqrt{1-\alpha^2}}{\alpha}$. It yields
\begin{equation*}
    \left\Vert \hat f\vert_{\mathcal{R}} \right\Vert_{L^2(\mathcal{R}, d\mu)}^2 \leq C \itg_0^1 \bigg( \itg_{R^3_x} \vert \Grave{f} (x,\alpha) \vert^p dx \bigg)^{\frac 2 p} \bigg( \frac{\sqrt{1-\alpha^2}}{\alpha}\bigg)^{-\frac 1 3} \frac{1}{\sqrt{1-\alpha^2}}d \alpha,
\end{equation*}
with $p = \frac 6 5 $, and $\Grave{f}$ being the partial Fourier transform of $f$ with respect to the $\alpha$ coordinate. 

Then, using Minskowski inequality to bound the $L^2_\alpha L^p_x$ norm by an $L^p_x L^2_\alpha$ norm, we have the following bound 
\begin{equation}
    \label{fgrave_mink}
    \left\Vert \hat f\vert_{\mathcal{R}} \right\Vert_{L^2(\mathcal{R}, d\mu)}^2 \leq C  \bigg( \itg_{R^3_x} \bigg( \itg_0^1 \alpha^{\frac 1 3} (1-\alpha^2)^{-\frac 2 3}\vert \Grave{f} (x,\alpha) \vert^2 d\alpha \bigg)^{\frac p 2} dx \bigg)^{\frac 2 p}.
\end{equation}
Now, we can split the following integral to isolate the singular values $\alpha = 0$ and $\alpha = 1$.
$$     \itg_0^1 \alpha^{\frac 1 3} (1-\alpha^2)^{-\frac 2 3}\vert \Grave{f} (x,\alpha) \vert^2 d\alpha = \itg_0^{\frac 1 2} ... + \itg_{\frac 1 2}^1 ...\ .
$$
We study the Sobolev weights in $\alpha$ in each piece, as we did in the proof of Theorem \ref{th:pri}. For the first integral piece, we apply a blunt bound on the weight $\alpha^{\frac 1 3} \leq \alpha^{-\frac 2 3} $ to bound the first integral by the Sobolev norm $ \Vert \Grave{f}(x,\cdot) \Vert^2_{\dot H^{- \frac 1 3}(\R)}$. The second integral, however, already has the desired Sobolev weight: since $(1-\alpha^2)^{-\frac 2 3} = (1-\alpha)^{-\frac 2 3} (1+\alpha)^{-\frac 2 3} $ the weight at $\alpha = 1$ behaves like a $- \frac 2 3$ exponent ; and we can sharply bound this by the Sobolev norm $ \Vert \Grave{f}(x,\cdot) \Vert^2_{\dot H^{- \frac 1 3}(\R)}$. We thus have
$$     \itg_0^1 \alpha^{\frac 1 3} (1-\alpha^2)^{-\frac 2 3}\vert \Grave{f} (x,\alpha) \vert^2 d\alpha \leq C \Vert {f}(x,\cdot) \Vert_{\dot H^{- \frac 1 3}(\R)}^2,$$
and the dual Sobolev embedding $L^p(\R) \hookrightarrow H^{- \frac 1 3}(\R) $ combined with \eqref{fgrave_mink} directly leads to 
$$ \left\Vert \hat f\vert_{\mathcal{R}} \right\Vert_{L^2(\mathcal{R}, d\mu)}^2 \leq C \Vert f \Vert^2_{L^p(\R^4)},$$
which concludes the proof.
\end{proof}

\section{Euler equation in a rotational framework} 

\subsection{Linear analysis}

In this part, we aim at proving Theorem \ref{th:rot} stated in the introduction. To describe the behaviour of non-viscous rotating fluids, the natural system to look at is the following, which is just the Euler equation with an additional rotation term, written as \eqref{eulertourne} in the introduction. Let us rewrite this system for more clarity:
$${\left\{\begin{array}{rcl}
\partial_t u + u\cdot \nabla u + \dfrac{e_3 \wedge u}{\varepsilon} + \nabla p &=& 0 \\
\nabla \cdot u &=& 0 \\
u(0,\cdot)&=&u_0,
\end{array}\right.}
$$
where $u$ is the velocity, $p$ the pressure and ${\varepsilon}$ is the Rossby number (hence $\frac 1 \varepsilon$ is proportional to the rotation speed). This system is the same as the one studied in \cite{kohleetakada}, and similar to \cite{cdgg_2} and \cite{cdgg} but here there is no viscosity. As explained in the introduction, viscosity enables one to prove fruitful Strichartz estimates that in turn grant the Taylor-Proudman theorem.

There are two main difficulties in System \eqref{eulertourne}: the pressure, and the advection term. As we explained in the previous sections, the classical method to get rid of the unknown pressure field is to apply the Leray projector $\mathbb{P}$ onto divergence-free vector fields, resulting in the following equations:
$$     {\left\{\begin{array}{lrcl}
\partial_t u +  \mathbb{P}\ (\dfrac{e_3 \wedge u}{\varepsilon}) = -\mathbb{P}\ (u \cdot \nabla u)  \\
\text{div} \ u = 0. \end{array}\right.} $$

The nonlinear term is treated as usual as a forcing term, which leads us to study the semigroup $S(t)$ associated to the linear equation :

\begin{equation}
\label{eulersgl}
 \partial_t u  + \mathbb{P}\ \dfrac{e_3 \wedge u}{\varepsilon} = 0.
\end{equation}
Strichartz estimates for this semigroup were proved in \cite{dut} and \cite{kohleetakada}. They were in turn used to prove long-time existence of solutions to the Euler equations in the rotational framework. The viscous case was treated in \cite{cdgg} using the Coriolis force to obtain a dispersive estimate. The proof in \cite{dut} adapts those arguments to the non-viscous case, and improves on them by resorting to the Riesz-Thorin theorem ; whereas in \cite{kohleetakada} a slightly more precise study is made to obtain sharp results. 

Our goal here is to obtain estimates through a global restriction theorem on a suitable surface. To do so, we first need to write the solution of \eqref{eulersgl} as an inverse Fourier transform. We once again consider the case $\varepsilon=1$, as the general case can be easily recovered by a change of variables. We write the equation in Fourier space, which gives the linear ODE:
$$ \partial_t \hat{u} (t,\xi) = A(\xi) \hat{u} (t, \xi),$$
where $A(\xi)$ is a matrix given by $A(\xi)X = \frac{\xi_3 \xi \wedge X}{\vert\xi\vert^2}$, with eigenvalues $0$ and $ \pm i \frac{\xi_3}{\vert \xi \vert}$. The associated eigenvector to the eigenvalue 0 is $e_0(\xi) = \frac{1}{\vert\xi\vert}\xi$, and the other two eigenvectors (computed in \cite{cdgg}) are
$$e_{1,2}(\xi) = \frac{1}{\sqrt 2 \vert\xi \vert\ \vert\xi_h\vert}
(\xi_1\xi_3 \mp i \xi_2 \vert \xi \vert,
\xi_2\xi_3 \pm i \xi_1 \vert \xi \vert,
-\vert\xi_h\vert^2),$$
where $\xi_h = (\xi_1,\xi_2)$.

Using that $u_0$ is divergence free, we infer that $\widehat{u_0}$ has no component on $e_0$. By projecting $\widehat{u_0}(\xi)$ onto the two other eigenspaces, the semigroup $S(t)$ can be written in terms of the semigroups $e^{\pm it\frac{D_3}{\vert D \vert}}$, defined as:
\begin{equation}
\label{surface}
    e^{\pm i t  \frac{D_3}{\vert D \vert}}g :(t,x) \longmapsto \itg_{\R^3} e^{ix\cdot \xi \pm it \frac{\xi_3}{\vert\xi\vert}}\hat g (\xi) d\xi.
\end{equation}
Remark that this time there is no equivalent of a quasigeostrophic part that would be trivially propagated, and all components of $\hat u_0$ are subject to dispersion.
Now that we have identified the semigroup at play, we can obtain Strichartz estimates for the linear system using restriction theory.

\subsection{Proof of Theorem \ref{th:rot}}

The inhomogeneous estimate of Theorem \ref{th:rot} is a direct consequence of Lemma 5 once again, so we will only give details for the proof of the homogeneous inequality. 

The explicit formula for the semigroup $e^{\pm it\frac{D_3}{\vert D \vert}}$ involves an inverse Fourier transform on the surface
$$ \M:= \left\{ \bigg(\xi, \frac{\xi_3}{\vert \xi \vert}\bigg),\ \ \ \xi \in \R^3\right\} \subset \R^4.$$
Strichartz estimates for this semigroup stem from a restriction theorem on the surface $\M$, simply by using its dual statement, as it was done in the previous sections. Such a restriction estimate is given by the following lemma.

\begin{lemma}\label{lemmerot}
Consider the surface $\M$ endowed with the measure $$d\sigma\bigg(\xi,\frac{\xi_3}{|\xi|}\bigg) = \frac{d\xi}{|\xi|^2}\cdotp$$ 

There exists a constant $C>0$ such that for any $f\in\s(\R^4)$, one has
$$ \Vert \hat f\vert_{\M}  \Vert_{L^2(\M,d\sigma)} \leq C \Vert f \Vert_{L^{\frac{6}{5}}(\R^4)}. $$
\end{lemma}

\begin{proof}
Consider $ f\in\s(\R^{4}) $ and set $\xi_h:=(\xi_1,\xi_2)$. We begin by writing $ \Vert \hat f\vert_{\M}  \Vert_{L^2(\M,d\sigma)} $ as an integral over all level sets by changing the $\xi_3$ variable, letting $\mu=\dfrac{\xi_3}{\vert \xi\vert}$ (with $\xi_h$ fixed). By straightforward computations, we have the identities $d\xi_3=\vert \xi_h \vert (1-\mu^2)^{-\frac 3 2} d\mu $ and $\vert \xi \vert = \dfrac{\vert \xi_h\vert}{\sqrt{1-\mu^2}} $, which yield:
\begin{align*}
\Vert \hat f\vert_{\M}  \Vert^2_{L^2(\M,d\sigma)} &= \itg_{\R^3} \left\vert \hat f(\xi_h,\xi_3,\frac{\xi_3}{\vert \xi \vert } )\right\vert^2 \ \ \vert \xi\vert^{-2} d\xi_hd\xi_3 \\ &= \int_{-1}^{1}  \itg_{\R^2_{\xi_h}} \left\vert \hat f(\xi_h,  \dfrac{\mu}{\sqrt{1-\mu^2}} \vert \xi_h \vert , \mu )\right\vert^2 \ \frac{d\xi_h}{\vert \xi_h \vert} \ \dfrac{1}{\sqrt{1-\mu^2}} d\mu .
\end{align*}
We see that at a given level $\mu$, the integral over $\xi_h$ is simply a squared $L^2$ norm on a cone with the appropriate measure $\dfrac{d\xi_h}{\vert \xi_h \vert}$. We can apply Corollary 1 with $\rho=\dfrac{\mu}{\sqrt{1-\mu^2}}$ to get
\begin{align*}
\Vert \hat f\vert_{\M}  \Vert^2_{L^2(\M,d\sigma)} &\leq C\int_{-1}^{1}  \bigg( \itg_{\R^3_x} \left\vert  \Grave{f}(x, \mu )  \right\vert^{p} dx \bigg)^{\frac{2}{p}} \ \vert\mu\vert^{-\frac 1 3}(1-\mu^2)^{-\frac 1 3} d\mu ,
\end{align*}
with $p=\frac 6 5$, and $\Grave{f}(x,\mu)$ the Fourier transform of $f$ in the last coordinate $x_4$ as defined in \eqref{fgrave}.

As before, we now bound the $L^{2}_{\mu}L^{p}_{x}$ norm by a $L^{p}_{x}L^{2}_{\mu}$ norm using the Minkowski inequality. We are left with
\begin{align*}
\Vert \hat f\vert_{\M}  \Vert^2_{L^2(\M,d\sigma)} &\leq C  \bigg(\itg_{\R^3_x}\bigg( \int_{-1}^{1} \left\vert  \Grave{f}(x, \mu )  \right\vert^{2} \vert\mu\vert^{-\frac 1 3}(1-\mu^2)^{-\frac 1 3} d\mu \bigg)^{\frac{p}{2}} \  dx\bigg)^{\frac{2}{p}}.
\end{align*}
We now want to use Sobolev embeddings for the inner integral to bound it by a $L^{p}$ norm. Note that we have three singularities in $\mu=-1$, $0$, $1$, however near each of those points the term $\vert\mu\vert^{-\frac 1 3}(1-\mu^2)^{-\frac 1 3}=\vert\mu\vert^{-\frac 1 3}(1-\mu)^{-\frac 1 3}(1+\mu)^{-\frac 1 3}$ behaves like a simple Sobolev weight. We make use of this by splitting the integral in three:
$$\int_{-1}^{1} \left\vert  \Grave{f}(x, \mu )  \right\vert^{2} \vert\mu\vert^{-\frac 1 3}(1-\mu^2)^{-\frac 1 3} d\mu = \int_{-1}^{-\frac 1 2} ... + \int_{-\frac 1 2}^{\frac 1 2} ... + \int_{\frac 1 2}^{1} ...\ .$$
Let us bound the last term, the others being dealt with identically: we simply use that, for $\frac 1 2 \leq \mu \leq 1$, we have $\vert\mu\vert^{-\frac 1 3}(1-\mu)^{-\frac 1 3}(1+\mu)^{-\frac 1 3}\leq C (1-\mu)^{-\frac 1 3}$. This yields, making an appropriate change of variables:
\begin{align*}
\int_{\frac 1 2}^{1}\left\vert  \Grave{f}(x, \mu )  \right\vert^{2} \vert\mu\vert^{-\frac 1 3}(1-\mu^2)^{-\frac 1 3} d\mu &\leq C\int_{\frac 1 2}^{1}\left\vert  \Grave{f}(x, \mu )  \right\vert^{2}  (1-\mu)^{-\frac 1 3} d\mu \\ &=C\int_{0}^{\frac 1 2}\left\vert  \Grave{f}(x, 1-\nu )  \right\vert^{2}  \nu^{-\frac 1 3} d\nu .
\end{align*}
The Sobolev exponent we seek is $s=-\frac{1}{3}$, to exploit that $L^{p}$ is embedded in $\dot{H}^{s}(\R)$. Since $\vert \nu \vert \leq 1$, we can bound $\nu^{-\frac 1 3}$ by $\nu^{-\frac 2 3}$. Unlike the previous sections, at all the singularities of the surface, we are forced to apply a very blunt bound. Whereas the slicing method could capture optimal results for Theorems \ref{th:pri} and \ref{th:bou}, we must lose the sharpness with this surface. What happens if we avoid using this blunt bound is discussed in Section 5.3. Continuing with the proof and using this bound anyway, we get
\begin{align*}
\int_{\frac 1 2}^{1}\left\vert  \Grave{f}(x, \mu )  \right\vert^{2} \vert\mu\vert^{-\frac 1 3}(1-\mu^2)^{-\frac 1 3} d\mu &\leq C\int_{0}^{\frac 1 2}\left\vert  \Grave{f}(x, 1-\nu )  \right\vert^{2}  \nu^{-\frac 2 3} d\nu \\&\leq C\int_{\R_{\nu}}\left\vert  \Grave{f}(x, 1-\nu )  \right\vert^{2}  \vert\nu\vert^{-\frac 2 3} d\nu\\&\leq C \Vert f(x,\cdot )\Vert ^{2}_{L^{p}(\R)}.
\end{align*}
The same bound holds for the two other terms, so that the whole integral over $\mu$ is controlled by $\Vert f(x,\cdot)\Vert ^{2}_{L^{p}(\R)}$, which in turns yields
\begin{align*}
\Vert \hat f\vert_{\M}  \Vert^2_{L^2(\M,d\sigma)} &\leq C  \bigg(\itg_{\R^3_x}\bigg(  \Vert f(x,\cdot) \Vert^{2}_{L^{p}(\R)}\bigg)^{\frac{p}{2}} \  dx\bigg)^{\frac{2}{p}}\\&= C  \Vert f\Vert ^{2}_{L^{p}(\R^4)}.
\end{align*}
This concludes the proof of Lemma 4, and therefore Theorem \ref{th:rot}.
\end{proof}

\subsection{Discussion of the sharpness}

Let us discuss the optimality of the result. The Strichartz inequalities obtained for the Euler equations using the slicing method are not sharp. More precisely, the optimal range of estimates obtained in \cite{kohleetakada} is
$$ \Vert e^{it\frac{D_3}{\vert D \vert}} u_0 \Vert_{L^q_{t}(\R, L^r_{x}(\R^d))} \leq C \Vert u_0 \Vert_{\dot H^{\frac{3}{2}-\frac{3}{r}}(\R^d)}, $$
for any $(q,r)\neq(2,\infty)$ such that $\frac{1}{q}+\frac{1}{r}\leq \frac{1}{2}$. Our technique only allows us to obtain, after interpolation with energy conservation, the range $\frac{2}{q}+\frac{1}{r}\leq \frac{1}{2}$, for $q\geq 6$.

However, it is easy to identify where the sharpness is lost in the obtention of Theorem \ref{th:rot}: as mentioned in the proof, we used an unnecessarily blunt bound $\nu^{-\frac{1}{3}} \leq \nu^{-\frac{2}{3}}$, because the Sobolev embedding $L^{\frac 6 5}(\R) \hookrightarrow \dot H^{-\frac 1 3}(\R)$ is the one that provides a bound by an $L^{\frac 6 5}$ norm in the last coordinate, so that we obtain in the end the desired $L^{\frac 6 5}$ norm on the whole space $\R^4$. For the primitive system and Boussinesq equation, one of the singularity naturally exhibits the correct Sobolev weight, and this is why the results are sharp. However, in this case, the bound makes us ignore some of the curvature of the surface in the vertical direction, which could be used to improve the exponents. To avoid this, we could have kept the term $\nu^{-\frac 1 3}$, and instead use the Sobolev embedding $L^{\frac 3 2}(\R) \hookrightarrow \dot H^{-\frac 1 6}(\R)$. However, at the end, we would have got the following restriction result involving anisotropic Lebesgue norms:
\begin{equation}
    \label{toto}
     \Vert \hat f\vert_{\M}  \Vert_{L^2(\M,d\sigma)} \leq C \Vert f \Vert_{L^{\frac{6}{5}}(\R^{3}_{x_1,x_2,x_3},L^{\frac{3}{2}}(\R_{x_4}))}.
\end{equation}
In terms of estimates for the semigroup $e^{it\frac{D_3}{\vert D \vert}}$, it writes 
$$\Vert e^{it\frac{D_3}{\vert D \vert}} u_0 \Vert_{L^6_{x}L^3_{t}} \leq C \Vert u_0 \Vert_{\dot H^1(\R^3)}.$$
This is not really a Strichartz estimate, we would have wished for a $L^3_t L^6_x$ norm instead, but to our knowledge there is no way around this issue (see \cite{bbg} where the same phenomenon occurs). We know thanks to \cite{kohleetakada} that $$ \Vert e^{it\frac{D_3}{\vert D \vert}} u_0 \Vert_{L^3_{t}L^6_{x}} \leq C \Vert u_0 \Vert_{\dot H^1(\R^3)} $$ holds and is sharp, so our result with inverted norms is just slightly weaker than ideal.

A way to try and obtain a sharp Strichartz estimate would be to aim for the same exponents in time and space (so that their order does not matter). This would be the following estimate $$ \Vert e^{it\frac{D_3}{\vert D \vert}} u_0 \Vert_{L^4_{t}L^4_{x}} \leq C \Vert u_0 \Vert_{\dot H^{\frac 3 4}(\R^3)} $$ from \cite{kohleetakada}. Those exponents are too strong to be achieved using cone restriction theorems applied to the level sets, since this would require Theorem 2 but with $p_0=\frac 4 3$ in dimension $d=2$ which is not possible, see \cite{tao}, which showcases a limit of the slicing method.

\begin{Remark}
As a final remark, one way to understand the origin of the difference between the primitive and Boussinesq systems and the rotating Euler equations is as follows: the Euler equations in a rotational framework can be seen as the $F\rightarrow \infty$ limit of the Primitive system, since they do not consider the external Earth gravity field. This limit-case point of view translates one-to-one to the singularities appearing in the last integral in the proofs of Theorem \ref{th:pri} and \ref{th:rot}: since plugging $F^{-1}=0$ in the integrand at the end of the proof of Theorem \ref{th:pri}, which us $\mu (1-\mu^2)^{-\frac 1 3}(\mu^2-F^{-2})^{-\frac 2 3}$, yields exactly the integrand in Theorem \ref{th:rot}, $\mu^{-\frac 1 3} (1-\mu^2)^{-\frac 1 3}$. The exponent at the singularity $0$ in the setting of rotating Euler equations can hence be seen as the combination of two contributions: an intrinsic $\mu$ term at $0$ —which also appears in the Boussinesq and Primitive systems— and the singularity at $F^{-1}$ that merges with the former term in the $F\rightarrow \infty$ limit. The divergence diminishes from a $-\frac{2}{3}$ exponent to $-\frac{1}{3}$, hence the possibility of improving the Strichartz estimate.

\end{Remark}

\begin{center}
\textsc{Acknowledgments}
\end{center}

We are very grateful to Isabelle Gallagher for all her advice during the writing of this article and for her ideas that gave birth to this work.  We also wish to thank Frédéric Charve for his additional precious advice, and finally the Département de Mathématiques et Applications of the ENS for their warm welcome.

\newpage

\section*{Appendix}

We prove for the convenience of the reader the following lemma, used to obtain the inhomogeneous Strichartz estimates.

\begin{lemma}
    Consider the abstract system
    \begin{equation}
    \label{system_lemma}
    \left\{\begin{array}{rcl}
\partial_t u + A u &=& \varphi \\
u\vert_{t=0} &=& u^0
 \end{array}\right.
\end{equation} 
Assume that the semigroup associated to the first equation of \eqref{system_lemma} satisfies the homogeneous estimate :
$$ \Vert e^{tA} u^0 \Vert_{L^p_t L^q_x} \leq C_0 \Vert u^0 \Vert_{\dot H^s}$$
for any $u^0$ in $\dot H^s$. Assume also that $(e^{At})^*=e^{-At}$.

Then, the solution of the inhomogeneous problem with vanishing initial data, which writes
$$v:t\longmapsto \itg_0^t e^{(t-t')A} \varphi(t') dt',$$
satisfies the inequality
$$ \Vert v \Vert_{L^p_t L^q_x} \leq C_0 \Vert\varphi \Vert_{L^1_t \dot H^{s}}.$$
\end{lemma}

\begin{proof}

Denote $B^{p,q}_1 = \{ h \in \mathcal S,\ \Vert h \Vert_{L^{p'}_t L^{q'}_x} \leq  1 \}$ the set of Schwartz functions in the unit ball of the dual space of $L^p_t L^q_x$, and denote $S(t)$ the semigroup $e^{At}$. Let $$v = \itg_\R \chi(t,t') S(t-t')\varphi(t')dt',$$
where $\chi(t,t') = {\left\{\begin{array}{lcl}
1\ \ \text{if} \ \ t'\in (0,t) \\
0 \ \ \text{otherwise} \end{array}\right.}$. 
Let $h \in B^{p,q}_1$. Standard computations give

\begin{align*}
    \itg_{\R_t \times \R^3_x} v \overline{h}\ dt dx & =  \itg_{\R_t \times \R^3_x} \itg_0^t S(t-t')  \varphi(t') \overline{h}(t) dt' dt dx \\
    &= \itg_{\R_t} \itg_0^t \bigg( S(t)S(-t')  \varphi(t') \vert h(t) \bigg)_{L^2} dt' dt \\
    &= \itg_{\R_t} \itg_{\R_{t'}} \bigg( S(t)S(-t')  \varphi(t') \vert \chi(t,t') h(t) \bigg)_{L^2} dt' dt \\
    &= \itg_{\R_{t'}} \bigg(  S(-t') \varphi(t') \vert  \itg_{\R_t} S(-t)\chi(t,t') h(t) dt \bigg)_{L^2} dt' .
\end{align*}
Now, notice that if we set $$\fonction{T}{\dot H^s}{L^p_t L^q_x}{u^0}{t\mapsto S(t)u^0},$$
the dual operator of $T$ is $$ \fonction{T^*}{L^{p'}_tL^{q'}_x}{\dot H^{-s}}{h}{\itg_\R S(-t) h(t) dt},$$
and its operator norm is also bounded by $C_0$. Therefore, since the $L^2$ scalar product is equal to the duality bracket for homogeneous Sobolev spaces, we have 

$$    \itg_{\R_{t'}} \bigg(  S(-t') \varphi(t') \vert  \itg_{\R_t} S(-t)\chi(t,t') h(t) dt \bigg)_{L^2} dt' \leq \itg_{\R_{t'}} \Vert S(-t') \varphi(t') \Vert_{\dot H^{s}} \Vert T^*(\chi(\cdot, t') h) \Vert_{\dot H^{-s}} dt'.
$$
By continuity of $T^*$ and since the semigroup $S$ is unitary on homogeneous Sobolev spaces, it yields
$$\itg_{\R_t \times \R^d_x} v \overline{h}\ dt dx  \leq  \itg_{\R_{t'}} \Vert  \varphi(t') \Vert_{\dot H^{s}} \Vert \chi(\cdot,t') h \Vert_{L^{p'} L^{q'}} dt'. $$
Now, as $\Vert \chi(\cdot,t') h \Vert_{L^{p'} L^{q'}} \leq \Vert h \Vert_{L^{p'} L^{q'}} \leq 1$, we have the desired inequality 
$$\itg_{\R_t \times \R^3_x} v \overline{h}\ dt dx  \leq \Vert \varphi \Vert_{L^1_t \dot H^s},  $$
which concludes the proof by dual characterization of the norm $\Vert v \Vert_{L^p_t L^q_x}$.

\end{proof}
\newpage

\begin{small}
\printbibliography[title = Bibliography]
\end{small}

\end{document}